\documentclass[12pt]{amsart}
\usepackage{graphicx}
\usepackage{amscd}
\usepackage[top=1in,bottom=1in,left=1.15in,right=1.15in]{geometry}
\usepackage{amsmath,amssymb,amsfonts}
\usepackage{amsthm}
\usepackage{color}
\usepackage{hyperref}
\usepackage{tikz-cd}
\usepackage{epstopdf}
\usetikzlibrary{matrix}
\usepackage{verbatim}
\usepackage{array}
\usepackage{mathtools}
\usepackage{enumitem}
\newtheorem{theo}{Theorem}[section]
\newtheorem{prop}[theo]{Proposition}
\newtheorem{lemma}[theo]{Lemma}
\newtheorem{defn}[theo]{Definition}
\newtheorem{rem}[theo]{Remark}

\newtheorem{ex}[theo]{Example}

\theoremstyle{definition}

\newtheorem{example}[theo]{Example}

\allowdisplaybreaks

\def\diaCrossP{\unitlength.08em
  \begin{minipage}{15\unitlength}
    \begin{picture}(15,15)
      \put(0,0){\vector(1,1){15}}
      \qbezier(15,0)(15,0)(10,5)
      \qbezier(5,10)(0,15)(0,15)
      \put(0,15){\vector(-1,1){0}}
    \end{picture}
  \end{minipage}
}

\def\diaCrossN{\unitlength.08em
  \begin{minipage}{15\unitlength}
    \begin{picture}(15,15)
      \put(15,0){\vector(-1,1){15}}
      \qbezier(0,0)(0,0)(5,5)
      \qbezier(10,10)(15,15)(15,15)
      \put(15,15){\vector(1,1){0}}
    \end{picture}
  \end{minipage}
}

\def\diaCircle{\unitlength.1em
  \begin{minipage}{15\unitlength}
    \begin{picture}(15,15)
      \put(7.5,10){\circle{8}}
      \put(7.5,0){\vector(0,1){5}}
       \put(7.5,5){\line(0,1){5.5}}
    \end{picture}
  \end{minipage}
}

\def\Ptwist{\unitlength.1em
  \begin{minipage}{15\unitlength}
    \begin{picture}(15,15)
     \put(0,7.5){\line(1,0){6}}
     \put(9,7.5){\line(1,0){6}}
      \put(4,9.5){\line(1,0){3}}
      \qbezier(7,9.5)(7,9.5)(9,5.5)
      \put(9, 5.5){\line(1,0){3}}
    \end{picture}
  \end{minipage}
}

\def\Ntwist{\unitlength.1em
  \begin{minipage}{15\unitlength}
    \begin{picture}(15,15)
     \put(0,7.5){\line(1,0){6}}
     \put(9,7.5){\line(1,0){6}}
      \put(4,5.5){\line(1,0){3}}
      \qbezier(7,5.5)(7,5.5)(9,9.5)
      \put(9, 9.5){\line(1,0){3}}
    \end{picture}
  \end{minipage}
}

\usetikzlibrary{decorations.markings}

\tikzset{->-/.style={decoration={
  markings,
  mark=at position .5 with {\arrow{>}}},postaction={decorate}}}

\tikzset{-<-/.style={decoration={
  markings,
  mark=at position .5 with {\arrow{<}}},postaction={decorate}}}

\begin{document}
\title{A multi-variable Alexander polynomial for a framed transverse graph}
\author{Yuanyuan Bao and Zhongtao Wu}
\address{Division of Mathematics \&
Research Center for Pure and Applied Mathematics,
Graduate School of Information Sciences,
Tohoku University, 6-3-09 Aramaki-Aza-Aoba, Aoba-ku, Sendai 980-8579, Japan}
\email{yybao@tohoku.ac.jp}
 \address{
Department of Mathematics, The Chinese University of Hong Kong, Shatin, Hong Kong
}
\email{ztwu@math.cuhk.edu.hk}

\begin{abstract}
We propose a definition of the rotation number for transverse graph diagrams, extending the classical notion of the rotation number for plane curves. Using this, we introduce a normalized multi-variable Alexander polynomial for framed, oriented transverse graphs without sinks or sources, 
embedded in the 3-sphere $S^3$. We prove that our invariant coincides with the $U_q(\mathfrak{gl}(1\vert 1))$-Alexander polynomial proposed by Viro.

\end{abstract}
\keywords{Alexander polynomial, transverse graph, rotation number}
\subjclass[2020]{Primary 57K10 05C10}

\maketitle

\tableofcontents

\section{Introduction}

The Alexander polynomial, originally defined for knots and links by J. W. Alexander in the 1920s, stands as one of the most fundamental invariants in low-dimensional topology. Over the decades, it has been generalized in numerous directions, including multi-variable versions for links, twisted Alexander polynomials, and extensions to spatial graphs.

In the context of spatial graphs, Viro \cite{MR2255851} constructed a multi-variable Alexander polynomial $\underline{\Delta}^1(\mathbb{G})$ using the representation theory of the quantum superalgebra $U_q(\mathfrak{gl}(1|1))$. While his invariant provides a powerful generalization for framed transverse graphs, it relies on sophisticated algebraic machinery. In our previous work \cite{MR4090586}, we developed a combinatorial approach to the Alexander polynomial for MOY graphs---transverse graphs with positive integer colorings---using Kauffman state sums. However, this construction was limited to single-variable polynomials obtained by specializing all variables through a homomorphism to $\mathbb{Z}$.

The present paper bridges these two approaches by extending our combinatorial construction to the multi-variable setting. A key inspiration comes from the classical theory of plane curves, where the \textit{rotation number} (or \textit{Whitney index}) measures the total turning number of an oriented immersed curve. Viro \cite{Viro1} provided an elegant formula that expresses this invariant as the sum of the winding numbers of the complementary regions minus the sum of the average winding numbers at self-intersection points. We extend this combinatorial perspective to transverse graph diagrams, defining a rotation number that serves as a crucial normalization factor in our construction.

Our main contribution is a diagrammatic formulation of a multi-variable Alexander polynomial $\Delta_{\mathbb{G}}$ for framed transverse graphs that provides a direct combinatorial interpretation of Viro's representation-theoretic invariant. This is accomplished through several key developments: extending Viro's rotation number formula to transverse graph diagrams (Definition~\ref{rotgraph}), constructing a normalized multi-variable state sum invariant (Definition~\ref{def:normalizedAlexanderpolynomial}, Theorem~\ref{thmDeltaG}), and proving the equivalence between our combinatorial approach and Viro's algebraic construction (Theorem~\ref{main}). Specifically, we establish that after a suitable change of variables, our polynomial and Viro's invariant are related by an explicit product factor over vertices of even type:
$$\Delta_{\mathbb{G}}(t_1^4, t_2^4, \ldots, t_k^4)=\prod_{\text{$v$: even type}} (t_v^{2}-t_v^{-2}) \cdot \underline{\Delta}^1(\mathbb{G}).$$

\subsection*{Organization} 
The paper is organized as follows. Section 2 develops the extended rotation number for transverse graph diagrams, generalizing Viro's formula from plane curves to diagrams of spatial graphs. Section 3 constructs the multi-variable Alexander polynomial via Kauffman states and explains how this invariant specializes to the original single-variable polynomial for MOY graphs. Section 4 establishes topological invariance, proving that our normalized Alexander polynomial is independent of the choice of base point and invariant under Reidemeister moves for framed transverse graphs. Section 5 reviews Viro's $U_q(\mathfrak{gl}(1|1))$-Alexander polynomial and proves the equivalence with our combinatorial invariant.

\subsection*{Acknowledgments}
We would like to thank Jiu Kang Yu for helpful discussions and suggestions. The first author is partially supported by a grant from TUMUG Support Program in Tohoku University. The second author is partially supported by a grant from the Research Grants Council of Hong Kong Special Administrative Region, China (Project No. 14301825) and a direct grant from CUHK (Project No. 4053717).

\section{Rotation numbers for transverse graph diagrams}

\subsection{Transverse spatial graphs}

We consider an oriented graph $G \subset S^3$ such that, 
for each vertex $v$, there exists a disk that separates the incoming 
and outgoing edges. We refer to such an orientation as a 
\emph{transverse orientation}, and a graph equipped with a transverse 
orientation will be called a \emph{transverse graph}. 
To the best of our knowledge, this terminology first appeared in 
\cite{MR3677933}, in the context of defining Heegaard Floer homology 
for graphs.

\begin{figure}[h!]
\begin{tikzpicture}[baseline=-0.65ex, thick, scale=1]
\draw (-1, -1.75) [->-] to (0, -0.75);
\draw (-0.5, -1.75) [->-] to (0, -0.75);
\draw (0, -0.75) [->] to (1, 0.25);
\draw (1, -1.75) [->-] to (0, -0.75);
\draw (0, -0.75) [->] to (-1, 0.25);
\draw (0, -0.75) [->] to (0.5, 0.25);
\draw (0, -0.75) node[circle,fill,inner sep=1.5pt]{};
\draw [dashed] (-0.7, -0.75)--(0.7, -0.75);
\draw (1.25, -0.75) node{$L_v$};
\draw (0.1, -1.5) node{$......$};
\draw (-0.1, 0) node{$......$};
\end{tikzpicture} \hspace{2cm}
\begin{tikzpicture}[baseline=-0.65ex, thick, scale=1]
\draw (0,-1)  to [out=90,in=270] (0.5,-0.33);
\draw (0,-1) to [out=270,in=180] (1.5,-2);
\draw (1.5,-2) to [out=0,in=0] (1.5,1);
\draw (0.5, -0.33) [->-] to (0.5,0.33);
\draw (0.5, 0.33) [->] to [out=90,in=0] (-0.5,1);
\draw (1,-1)  to [out=90,in=270] (0.5,-0.33);
\draw (1,-1) to [out=270,in=45] (0.6,-1.6);
\draw (0.3,-1.8) to [out=225,in=0] (-0.5,-2);
\draw (-0.5,-2) to [out=180,in=180] (-0.5,1);
\draw (1.5,1) [<-] to [out=180,in=90] (0.5,0.33);
\draw (0.5, -0.33) node[circle,fill,inner sep=1pt]{};
\draw (0.5, 0.33) node[circle,fill,inner sep=1pt]{};
\end{tikzpicture}
	\caption{The local picture of a vertex with transverse orientation (left). An oriented trivalent graph without sinks or sources is a transverse graph (right).}
\label{moygraph}
\end{figure}
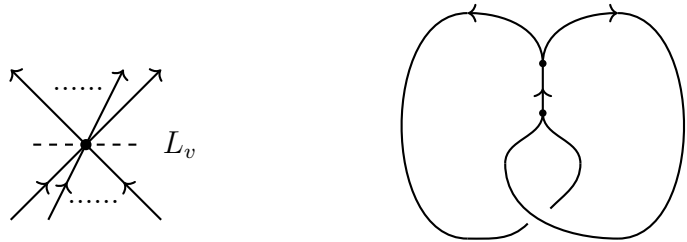

\begin{defn}
\rm 
A diagram $D$ of a transverse spatial graph $G$ on $\mathbb{R}^2$ is defined as a regular projection of $G$ satisfying the following conditions: 
\begin{enumerate}
  \item Self-intersections occur only as transverse double points between edges (called \emph{crossings}), where over/under information is specified.
  \item Around each vertex $v$, there exists a straight line $L_v$ separating the edges entering $v$ from those leaving $v$.
\end{enumerate}
When the position of the separating line $L_v$ is evident from the context, it may be omitted from the diagram.
\end{defn}

\medskip
\noindent
The following theorem is well-known to experts.
\begin{theo}
Two diagrams represent the same transverse graph if and only if they can be transformed into one another by a finite sequence of moves shown in Fig.~\ref{fig:e25}.
\end{theo}

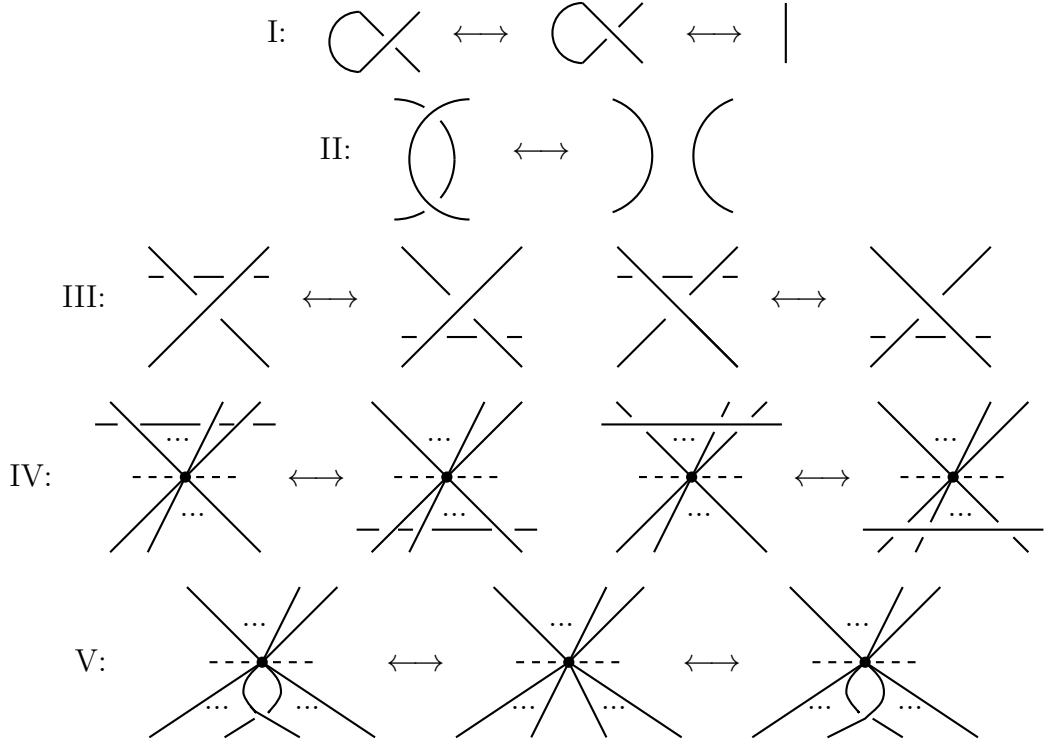
\begin{figure}[h!]
I: \quad \begin{tikzpicture}[baseline, thick, scale=0.4]
\draw (1,-1)   -- (0.2,-0.2);
\draw (-1, -1) -- (1, 1);
\draw (-0.2,0.2) --  (-1,1) ;
\draw (-1, 1) arc (90:270:1);
\end{tikzpicture}\quad  $\longleftrightarrow$ \quad
\begin{tikzpicture}[baseline=-0.65ex, thick, scale=0.4]
\draw (-1,-1)   -- (-0.2,-0.2);
\draw (1, -1) -- (-1, 1) ;
\draw (0.2,0.2)  --  (1,1);
\draw (-1, 1) arc (90:270:1);
\end{tikzpicture} \quad  $\longleftrightarrow$ \quad
\begin{tikzpicture}[baseline=-0.65ex, thick, scale=0.4]
\draw (0, -1) -- (0, 1) ;
\end{tikzpicture} \\  \vspace{3mm}
II: \quad \begin{tikzpicture}[baseline, thick, scale=0.4]
\draw (-1, 2) arc (90:270:2);
\draw (-3.5, 2) arc (90:60:2);
\draw (-1.5, 0) arc (0:40:2);
\draw (-1.5, 0) arc (0:-40:2);
\draw (-3.5, -2) arc (270:300:2);
\end{tikzpicture} \quad  $\longleftrightarrow$ \quad
\begin{tikzpicture}[baseline, thick, scale=0.4]
\draw (-1, 2) arc (110:250:2);
\draw (-5, 2) arc (70:-70:2);
\end{tikzpicture} \\  \vspace{3mm}
III: \quad\begin{tikzpicture}[baseline, thick, scale=0.4]
\draw (-2, 2) -- (-0.4, 0.4) ;
\draw (-2, 1) -- (-1.5, 1) ;
\draw (1.5, 1) -- (2, 1) ;
\draw (-0.5, 1) -- (0.5, 1) ;
\draw (0.4, -0.4) -- (2, -2) ;
\draw (2, 2) -- (-2, -2) ;
\end{tikzpicture}\quad  $\longleftrightarrow$ \quad
\begin{tikzpicture}[baseline, thick, scale=0.4]
\draw (-2, 2) -- (-0.4, 0.4) ;
\draw (-2, -1) -- (-1.5, -1) ;
\draw (1.5, -1) -- (2, -1) ;
\draw (-0.5, -1) -- (0.5, -1) ;
\draw (0.4, -0.4) -- (2, -2) ;
\draw (2, 2) -- (-2, -2) ;
\end{tikzpicture}\quad\quad\quad
\begin{tikzpicture}[baseline, thick, scale=0.4]
\draw (-2, 2) -- (2, -2) ;
\draw (-2, 1) -- (-1.5, 1) ;
\draw (1.5, 1) -- (2, 1) ;
\draw (-0.5, 1) -- (0.5, 1) ;
\draw (0.4, -0.4) -- (2, -2) ;
\draw (2, 2) -- (0.4, 0.4) ;
\draw (-2, -2) -- (-0.4, -0.4) ;
\end{tikzpicture}\quad  $\longleftrightarrow$ \quad
\begin{tikzpicture}[baseline, thick, scale=0.4]
\draw (-2, 2) -- (2, -2) ;
\draw (-2, -1) -- (-1.5, -1) ;
\draw (1.5, -1) -- (2, -1) ;
\draw (-0.5, -1) -- (0.5, -1) ;
\draw (2, 2) -- (0.4, 0.4) ;
\draw (-2, -2) -- (-0.4, -0.4) ;
\end{tikzpicture}\\\vspace{4mm}
IV: \quad \begin{tikzpicture}[baseline=-0.65ex, thick, scale=1]
\draw (-1, -1) -- (0, 0);
\draw (-0.5, -1) -- (0, 0);
\draw (0, 0) -- (1, 1);
\draw (1, -1) -- (0, 0);
\draw (0, 0) -- (-1,1);
\draw (0, 0) -- (0.5, 1);
\draw (0, 0) node[circle,fill,inner sep=1.5pt]{};
\draw [dashed] (-0.7, 0)--(0.7, 0);
\draw (0.1, -0.5) node{$...$};
\draw (-0.1, 0.5) node{$...$};
\draw (-1.2, 0.7) -- (-0.9, 0.7);
\draw (0.9, 0.7) -- (1.2, 0.7);
\draw (0.65, 0.7) -- (0.45, 0.7);
\draw (-0.6, 0.7) -- (0.2, 0.7);
\end{tikzpicture}  $\longleftrightarrow$
\begin{tikzpicture}[baseline=-0.65ex, thick, scale=1]
\draw (-1, -1) --  (0, 0);
\draw (-0.5, -1) -- (0, 0);
\draw (0, 0) -- (1, 1);
\draw (1, -1)--  (0, 0);
\draw (0, 0) --  (-1,1);
\draw (0, 0) --  (0.5, 1);
\draw (0, 0) node[circle,fill,inner sep=1.5pt]{};
\draw [dashed] (-0.7, 0)--(0.7, 0);
\draw (0.1, -0.5) node{$...$};
\draw (-0.1, 0.5) node{$...$};
\draw (-1.2, -0.7) -- (-0.9, -0.7);
\draw (0.9, -0.7) -- (1.2, -0.7);
\draw (-0.65, -0.7) -- (-0.45, -0.7);
\draw (0.6, -0.7) -- (-0.2, -0.7);
\end{tikzpicture}\quad\quad
\begin{tikzpicture}[baseline=-0.65ex, thick, scale=1]
\draw (-1, -1) --  (0, 0);
\draw (-0.5, -1) -- (0, 0);
\draw (0, 0) -- (0.6, 0.6);
\draw (0.8, 0.8) --  (1, 1);
\draw (1, -1) --  (0, 0);
\draw (0, 0) -- (-0.6,0.6);
\draw (-0.8, 0.8) --  (-1,1);
\draw (0, 0) -- (0.3, 0.6);
\draw (0.4, 0.8) --  (0.5, 1);
\draw (0, 0) node[circle,fill,inner sep=1.5pt]{};
\draw [dashed] (-0.7, 0)--(0.7, 0);
\draw (0.1, -0.5) node{$...$};
\draw (-0.1, 0.5) node{$...$};
\draw (-1.2, 0.7) -- (1.2, 0.7);
\end{tikzpicture}  $\longleftrightarrow$
\begin{tikzpicture}[baseline=-0.65ex, thick, scale=1]
\draw (-1, -1) -- (-0.8, -0.8);
\draw (-0.6, -0.6) --  (0, 0);
\draw (-0.5, -1) -- (-0.4, -0.8);
\draw (-0.3, -0.6) --  (0, 0);
\draw (0, 0) -- (1, 1);
\draw (0.6, -0.6) -- (0, 0);
\draw (1, -1) -- (0.8, -0.8);
\draw (0, 0) -- (-1,1);
\draw (0, 0) -- (0.5, 1);
\draw (0, 0) node[circle,fill,inner sep=1.5pt]{};
\draw [dashed] (-0.7, 0)--(0.7, 0);
\draw (0.1, -0.5) node{$...$};
\draw (-0.1, 0.5) node{$...$};
\draw (-1.2, -0.7) -- (1.2, -0.7);
\end{tikzpicture}\\\vspace{4mm}
V: \quad \begin{tikzpicture}[baseline=-0.65ex, thick, scale=1]
\draw (-1.5, -1) -- (0, 0);
\draw (0, 0) -- (1, 1);
\draw (1.5, -1) -- (0, 0);
\draw (0, 0) -- (-1,1);
\draw (0, 0) -- (0.5, 1);
\draw (0, 0) node[circle,fill,inner sep=1.5pt]{};
\draw [dashed] (-0.7, 0)--(0.7, 0);
\draw (0.6, -0.6) node{$...$};
\draw (-0.6, -0.6) node{$...$};
\draw (-0.1, 0.5) node{$...$};
\draw (0,0) to [out=225,in=90] (-0.25,-0.4) to [out=270,in=315] (-0.1,-0.65) to (0.5, -1);
\draw (-0.5, -1) -- (-0.1, -0.75);
\draw (0,0) to [out=315,in=90] (0.25,-0.4) to [out=270,in=225] (0.1,-0.65);
\end{tikzpicture}
$\longleftrightarrow$
\begin{tikzpicture}[baseline=-0.65ex, thick, scale=1]
\draw (-1.5, -1) -- (0, 0);
\draw (-0.5, -1) -- (0, 0);
\draw (0.5, -1) -- (0, 0);
\draw (0, 0) -- (1, 1);
\draw (1.5, -1) -- (0, 0);
\draw (0, 0) -- (-1,1);
\draw (0, 0) -- (0.5, 1);
\draw (0, 0) node[circle,fill,inner sep=1.5pt]{};
\draw [dashed] (-0.7, 0)--(0.7, 0);
\draw (0.6, -0.6) node{$...$};
\draw (-0.6, -0.6) node{$...$};
\draw (-0.1, 0.5) node{$...$};
\end{tikzpicture}$\longleftrightarrow$
\begin{tikzpicture}[baseline=-0.65ex, thick, scale=1]
\draw (-1.5, -1) -- (0, 0);
\draw (0, 0) -- (1, 1);
\draw (1.5, -1) -- (0, 0);
\draw (0, 0) -- (-1,1);
\draw (0, 0) -- (0.5, 1);
\draw (0, 0) node[circle,fill,inner sep=1.5pt]{};
\draw [dashed] (-0.7, 0)--(0.7, 0);
\draw (0.6, -0.6) node{$...$};
\draw (-0.6, -0.6) node{$...$};
\draw (-0.1, 0.5) node{$...$};
\draw (0,0) to [out=315,in=90] (0.25,-0.4) to [out=270,in=45] (0,-0.75) to (-0.5, -1);
\draw (0.5, -1) -- (0.1, -0.75);
\draw (0,0) to [out=225,in=90] (-0.25,-0.4) to [out=270,in=315] (-0.1,-0.65);
\end{tikzpicture}
\caption{Reidemeister moves for transverse graph diagrams. Suppressed orientations of the edges can be added in all compatible ways.}
\label{fig:e25}
\end{figure}

For a transverse graph $G$, let $S^3 \setminus G$ denote its complement in $S^3$. 
In this paper, we consider only transverse graphs $G$ where the meridian of each edge represents a \textit{nontrivial} element of $H_1(S^3 \setminus G;\mathbb{Z})$. This condition ensures that every edge of $G$ lies in at least one cycle.

\subsection{Extended rotation number}

Throughout this paper, a diagram $D$ is called {\it connected} if it is connected as a subspace of $\mathbb{R}^2$, i.e., the projection of the graph forms a connected set in the plane. This is distinct from the connectivity of the underlying graph $G \subset S^3$ itself. For example, a standard Hopf link diagram is connected in $\mathbb{R}^2$ even though the link is disconnected in $S^3$. 
This distinction is important for our definitions of the rotation number and state sum, which depend on the connectivity of the diagram as a planar object rather than the spatial connectivity of the graph.

\subsubsection{Classical rotation number for plane curves}
We begin by recalling the classical rotation number for plane curves.
For a plane curve, the rotation number (or Whitney index) \cite{MR1556973} measures the total number of turns made while traveling along the curve. For a union $D$ of finitely many oriented plane curves, The rotation number $w(D)$ is the sum of the rotation numbers of its components.

Among various studies, Viro \cite{Viro1} provided a formula for computing $w(D)$ in terms of the regions and double points associated with $D$. The formula is described as follows.

\medskip
For each region $r$ of $\mathbb{R}^2 \setminus D$, define $\chi(r)$ to be its winding number. Specifically, the unbounded region has winding number $0$, and $\chi(r)$ of the remaining regions are determined by the rule: when crossing a curve from left to right (with respect to the orientation of the curve), the winding number decreases by $1$.

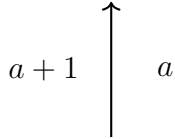
\begin{center}
\begin{tikzpicture}[baseline=-0.65ex, thick, scale=1.8]
\draw (0,0) [->] to (0,1);
\draw (0.4, 0.5) node {$a$};
\draw (-0.5, 0.5) node {$a+1$};
\end{tikzpicture}
\end{center}

For each double point $v$ of $D$, we define $\chi(v)$ as the average of the winding numbers of the four regions adjacent to $v$. More precisely, when the adjacent regions have winding numbers $a$, $a+1$, $a+1$, and $a+2$,  assign $\chi(v)=a+1$, as illustrated below.
\begin{figure}[h!]
\begin{tikzpicture}[baseline=-0.65ex, thick, scale=1]
\draw (-1,-1) [->] to (1,1);
\draw (1,-1) [->] to (-1,1);
\draw (1, 0) node {$a$};
\draw (-1, 0) node {$a+2$};
\draw (0, 1) node {$a+1$};
\draw (0, -1) node {$a+1$};
\end{tikzpicture}
\end{figure}

\begin{theo}[Viro \cite{Viro1}]
\label{Viroformula}
Assume that $D$ is connected as a subspace of $\mathbb{R}^2$. Then the rotation number of $D$ is given by
\[
w(D) = \sum_{r} \chi(r) - \sum_{v} \chi(v),
\]
where the first sum is taken over all regions $r$ of $\mathbb{R}^2 \setminus D$, and the second sum is taken over all double points $v$ of $D$.
\end{theo}

\medskip

\subsubsection{Generalization to transverse graphs}
Motivated by Viro's formula, we extend the rotation number to transverse graph diagrams. 

\medskip
Let $G$ be a transverse graph in $S^3$. We begin with a standard presentation of the first homology group $H_1(S^3 \setminus G; \mathbb{Z})$, which is a finitely-generated free abelian group.  
\begin{enumerate}
  \item \textit{Generators:} To each edge $e$ of $G$ (including loops), assign a generator $t_e$ corresponding to the homology class of its oriented meridian. 
  
  \item \textit{Relators:} All generators \textit{commute}, i.e., $ts = st$ for any generators $s$ and $t$. Additionally, at each vertex $v$ of $G$, if $s_1, s_2, \dots, s_k$ are the generators for incoming edges and $t_1, t_2, \dots, t_l$ for outgoing edges, then:
  \[
    s_1 s_2 \cdots s_k = t_1 t_2 \cdots t_l.
  \]
\end{enumerate}

\begin{rem}
While the first homology group is conventionally expressed additively as an abelian group, we intentionally use multiplicative notation with the relation $ts = st$ to indicate commutativity. This choice of notation will facilitate later sections where we define the Alexander polynomial using group rings.
\end{rem}

Consider a diagram $D$ of a transverse graph. To each connected component $r$ of $\mathbb{R}^2 \setminus D$ (called a {\it regular region}), we assign a value $\chi(r)\in H_1(S^3 \setminus G; \mathbb{Z})$, which we call {\it the winding number} of $r$. This notation extends the classical winding number, and it is defined as follows:
\begin{enumerate}
  \item The unique unbounded region has winding number $1 \in H_1(S^3 \setminus G; \mathbb{Z})$.
  
  \item For the remaining regions, $\chi(r)$ are determined inductively according to the rule illustrated below: 
  if an edge $e$ points upward and its left-hand side region has winding number $x \in H_1(S^3 \setminus G; \mathbb{Z})$, then the right-hand side region has winding number $x \cdot t_e$, where $t_e$ denotes the oriented meridian of the edge $e$.
\end{enumerate}

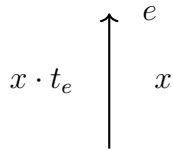
\begin{figure}[h!]
\centering
\begin{tikzpicture}[baseline=-0.65ex, thick, scale=1.8]
  \draw[->] (0,0) -- (0,1);
  \node at (0.3,1) {$e$};
  \node at (0.4,0.5) {$x$};
  \node at (-0.5,0.5) {$x \cdot t_e$};
\end{tikzpicture}
\caption{Inductive rule for assigning winding numbers to regions across an edge $e$.}
\end{figure}

For each vertex $v$ of $D$, we define a value $\chi(v)$ as follows. Let $s_1, s_2, \dots, s_k$ and $t_1, t_2, \dots, t_l$ denote the generators corresponding to the incoming and outgoing edges around $v$ respectively, as shown in Fig.~\ref{f4}.
Let $x_i$ denote the winding number of the region between $s_i$ and $s_{i+1}$ ($1 \le i \le k-1$), and $y_j$ the winding number of the region between $t_j$ and $t_{j+1}$ ($1 \le j \le l-1$). When $k=1$ or $l=1$, there are no corresponding $x_i$ or $y_j$.

\begin{figure}[h!]
\centering
\begin{tikzpicture}[baseline=-0.65ex, thick, scale=1.5] 
\draw (-1, -1.25) node[below]{$s_1$}[->-] to (0, -0.25); 
\draw (-0.5, -1.25)node[below]{$s_2$} [->-] to (0, -0.25); 
\draw (0, -0.25) [->] to (1, 0.75)node[above]{$t_l$}; 
\draw (1, -1.25) node[below]{$s_k$} [->-] to (0, -0.25); 
\draw (0, -0.25) [->] to (-1, 0.75) node[above]{$t_1$}; 
\draw (0, -0.25) [->] to (-0.5, 0.75)node[above]{$t_2$}; 
\draw (0, -0.25) node[circle,fill,inner sep=1.5pt]{}; 
\draw [dashed] (-0.7, -0.25)--(0.7, -0.25); \draw (1.25, -0.25) node{$L_v$}; 
\draw (-0.65, 0.6) node {$\color{gray}{x_1}$};
\draw (0, 0.6) node {$\color{gray}{x_2\cdots}$};
\draw (-0.65, -1.15) node {$\color{gray}{y_1}$};
\draw (0, -1.15) node {$\color{gray}{y_2\cdots}$};\end{tikzpicture}
\caption{Incoming and outgoing edges around a vertex $v$. Gray letters indicate the winding numbers of the regions between consecutive edges.}
\label{f4}
\end{figure}
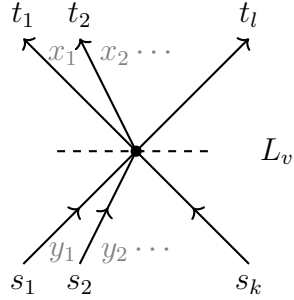

Then define
\[
\chi(v) := \left(\prod_{i=1}^{k-1} x_i^{1/2}\right)\cdot \left(\prod_{j=1}^{l-1} y_j^{1/2}\right),
\]
i.e., the product is taken over all regions around $v$ except the two regions that the separating line $L_v$ intersects.

For each double point $v$ of $D$, treating it as a vertex with two incoming and two outgoing edges, we define: 
\[
\chi(v) = x^{1/2} y^{1/2},
\] 
where $x$ and $y$ are the winding numbers of the south and north corner regions respectively.

\begin{lemma}
Suppose $D$ is a connected diagram of a transverse spatial graph $G$. Let $V$ be the set of vertices of $D$, and let $X(D)$ denote the set of crossings (double points) of $D$. 
Then the product
\[
\prod_{v \in V \cup X(D)} \chi(v)
\] 
is an element of $H_1(S^3 \setminus G; \mathbb{Z})$.
\end{lemma}

\begin{proof}
It suffices to show that for each region $r$ of $\mathbb{R}^2 \setminus D$, the factor $\chi(r)^{1/2}$ appears an even number of times in 
$
\prod_{v \in V \cup X(D)} \chi(v).
$

Consider the boundary $\partial r$ of region $r$ as a closed curve (possibly with repeated vertices and edges). At each vertex $v$ on $\partial r$, the boundary passes through two edges incident to $v$ that are consecutive in the cyclic order around $v$ with respect to region $r$. The contribution of $r$ to $\chi(v)$ is $\chi(r)^{1/2}$ if both edges are incoming to $v$ or both are outgoing from $v$, and $1$ otherwise (i.e., if one edge is incoming and the other is outgoing).  Note that at vertices where the boundary curve has a self-intersection, $r$ may appear in multiple corners around $v$, and each such occurrence is counted separately.

Now observe that each edge on the boundary $\partial r$ contributes once as an incoming edge at its head and once as an outgoing edge at its tail.  Therefore, the total number of incoming edges equals outgoing edges, forcing the number of vertices with both edges incoming to equal those with both outgoing. Consequently, $\chi(r)^{1/2}$ appears an even number of times in the product.


\end{proof}

Now we extend Viro's formula in Theorem \ref{Viroformula} to a transverse graph diagram.

\begin{defn}
\label{rotgraph}
\rm 
Let $D$ be a connected diagram of a transverse graph. The {\it rotation number} of $D$ is defined as  
\[
\mathrm{Rot}(D) \;:=\; \prod_{r \in F(D)} \chi(r) \; \bigg[ \prod_{v \in V \cup X(D)} \chi(v) \bigg]^{-1},
\]
where $F(D)$ denotes the set of regular regions of $\mathbb{R}^2 \setminus D$, and $X(D)$ denotes the set of crossings of $D$.   

For a disconnected diagram, the rotation number is defined as the product of the rotation numbers of its connected components.
\end{defn}

\begin{ex}
\rm
Consider the diagram $D$ in Fig.~\ref{example}.  Let $t$ and $s$ denote the generators corresponding to the nearby edges.
Gray letters indicate the winding numbers of the corresponding regions.
There are four regions, two vertices, and one crossing.  

We compute
\[
\prod_{r \in F(D)} \chi(r) \;=\; t \cdot s^{-1} \cdot ts^{-1} \;=\; t^{2}s^{-2},
\]
and
\[
\prod_{v \in V \cup X(D)} \chi(v) \;=\; (ts^{-1})^{1/2} \cdot (ts^{-1})^{1/2} \;=\; ts^{-1}.
\]
Therefore,
\[
\mathrm{Rot}(D) \;=\; t^{2}s^{-2} \,(ts^{-1})^{-1} \;=\; ts^{-1}.
\]
\end{ex}

 \begin{figure}[h!]
\begin{tikzpicture}[baseline=-0.65ex, thick, scale=1]
\draw (0,-1)  to [out=90,in=270] (0.5,-0.33);
\draw (0,-1) to [out=270,in=180] (1.5,-2);
\draw (1.5,-2) to [out=0,in=0] (1.5,1);
\draw (0.5, -0.33) [->-] to (0.5,0.33);
\draw (0.5, 0.33) [->] to [out=90,in=0] (-0.5,1);
\draw (1,-1)  to [out=90,in=270] (0.5,-0.33);
\draw (1,-1) to [out=270,in=45] (0.6,-1.6);
\draw (0.3,-1.8) to [out=225,in=0] (-0.5,-2);
\draw (-0.5,-2) to [out=180,in=180] (-0.5,1);
\draw (1.5,1) [<-] to [out=180,in=90] (0.5,0.33);
\draw (0,1.2) node {$t$};
\draw (1,1.2) node {$s$};
\draw (0.9,0) node {$ts$};
\draw (3,0) node {$\color{gray}{1}$};
\draw (1.6,-0.5) node {$\color{gray}{s^{-1}}$};
\draw (-0.6,-0.5) node {$\color{gray}{t}$};
\draw (0.5,-1) node {$\color{gray}{ts^{-1}}$};
\draw (0.5, -0.33) node[circle,fill,inner sep=1pt]{};
\draw (0.5, 0.33) node[circle,fill,inner sep=1pt]{};
\end{tikzpicture}
\caption{A diagram of a trivalent graph.}
\label{example}
\end{figure}
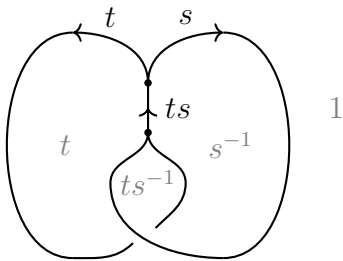

\medskip

\subsection{Properties of the rotation number}
We discuss some properties of the rotation number $\mathrm{Rot}(D)$.

\begin{prop}
\label{prop1}
The rotation number satisfies the following relations. In cases (iv), (v), and (vi), for diagrams $D$ and $D'$ corresponding to graphs $G$ and $G'$ respectively, if $H_1(S^3\setminus G; \mathbb{Z})$ embeds as a subgroup of $ H_1(S^3\setminus G'; \mathbb{Z})$, then we regard $\mathrm{Rot}(D)$ as an element of $H_1(S^3\setminus G'; \mathbb{Z})$ via this inclusion. 
\begin{enumerate}
\item
$\mathrm{Rot}\left(\begin{tikzpicture}[baseline=-0.65ex, thick, scale=0.8]
\draw (-1, -1.25) node[below]{$s_1$}[->-] to (0, -0.25);
\draw (-0.5, -1.25)node[below]{$s_2$} [->-] to (0, -0.25);
\draw (0, -0.25) [->] to (1, 0.75)node[above]{$t_l$};
\draw (1, -1.25) node[below]{$s_k$} [->-] to (0, -0.25);
\draw (0, -0.25) [->] to (-1, 0.75) node[above]{$t_1$};
\draw (0, -0.25) [->] to (-0.5, 0.75)node[above]{$t_2$};
\draw (0.1, -1) node{$......$};
\draw (0.1, 0.5) node{$......$};
\end{tikzpicture}\right)=\mathrm{Rot}\left(\begin{tikzpicture}[baseline=-0.65ex, thick, scale=0.8]
\draw (-1, -1.5) node[below]{$s_1$}[->-] to (0, -0.5);
\draw (-0.5, -1.5)node[below]{$s_2$} [->-] to (0, -0.5);
\draw (0, 0) [->] to (1, 1)  node[above]{$t_l$};
\draw (1, -1.5) node[below]{$s_k$} [->-] to (0, -0.5);
\draw (0, 0) [->] to (-1, 1)  node[above]{$t_1$};
\draw (0, 0) [->] to (-0.5, 1)  node[above]{$t_2$};
\draw (0, -0.5) [->-] to (0, 0);
\draw (0.1, -1.25) node{$......$};
\draw (0.1, 0.75) node{$......$};
\end{tikzpicture}\right)
$
\item $\mathrm{Rot} \left(\begin{tikzpicture}[baseline=-0.65ex, thick, scale=0.8]
\draw (0, 0) [->] to (1, 1) node[above]{$t_l$};
\draw (0, 0) [->] to (-1, 1) node[above]{$t_1$};
\draw (0, 0) [->] to (-0.5, 1) node[above]{$t_2$};
\draw (0, -1)node[below]{$t_1t_2 \cdots t_l$} [->-] to (0, 0);
\draw (0.1, 0.75) node{$......$};
\end{tikzpicture}\right)=\mathrm{Rot}\left(\begin{tikzpicture}[baseline=-0.65ex, thick, scale=0.8]
\draw (0, -0.5) [->] to (1, 1) node[above]{$t_l$};
\draw (0, -0.5) [->] to (-1, 1) node[above]{$t_1$};
\draw (0.25, -0.1) [->] to (-0.5, 1) node[above]{$t_2$};
\draw (0, -1.25) node[below]{$t_1t_2 \cdots t_l$} [->-] to (0, -0.5);
\draw (0.1, 0.75) node{$......$};
\end{tikzpicture}\right)$
\item $\mathrm{Rot}\left(
\begin{minipage}{0.2\linewidth}
\hspace*{0.2\linewidth}
   \rule[2cm]{0.7cm}{-5cm}
\begin{tikzpicture}[baseline=-0.65ex, thick, scale=0.8]
\draw (-1, -1.5)node[below]{$s_1$} [->-] to (0, -0.5);
\draw (-0.5, -1.5)node[below]{$s_2$} [->-] to (0, -0.5);
\draw (1, -1.5)node[below]{$s_k$} [->-] to (0, -0.5);
\draw (0, -0.5) [->-] to (0, 0.5) node[above]{$s_1s_2 \cdots s_k$};
\draw (0.1, -1.25) node{$......$};
\end{tikzpicture}
\end{minipage}
\right )=\mathrm{Rot}\left(
\begin{minipage}{0.2\linewidth}
\hspace*{0.2\linewidth}
   \rule[2cm]{0.7cm}{-5cm}
\begin{tikzpicture}[baseline=-0.65ex, thick, scale=0.8]
\draw (-1, -1.5) [->-] node[below]{$s_1$} to (0, -0.5);
\draw (-0.5, -1.5) [->-] node[below]{$s_2$} to (0.25, -0.75);
\draw (1, -1.5) [->-] node[below]{$s_k$} to (0, -0.5);
\draw (0, -0.5) [->-] to (0, 0.5) node[above]{$s_1s_2 \cdots s_k$};
\draw (0.25, -1.5) node{$......$};
\end{tikzpicture}
\end{minipage}
\right )$
\item $\mathrm{Rot} \left(\begin{tikzpicture}[baseline=-0.65ex,thick, scale=0.7]
\draw (0,-1.5) [->-] to (0,-0.5);
\draw (0,-0.5) [->-] to (0,0.5);
\draw (0,-0.5) [->] to (0,1.5);
\draw (0,-0.5) to [out=270, in=270] (1,-0.2) [-<-] to [out=90, in=270] (1,0.2) to [out=90, in=90] (0,0.5);
\draw (-0.3, -1) node {$t$};
\draw (-0.3, 1) node {$t$};
\draw (-0.6, 0) node {$ts$};
\draw (1.3,0) node {$s$};
\end{tikzpicture}\right) =s^{-1}\,\,\mathrm{Rot} \left(\begin{tikzpicture}[baseline=-0.65ex,thick, scale=0.7]
\draw (0,-1.5) [->] to (0,1.5);
\draw (-0.25,0) node {$t$};
\draw (0.25,0) node {};
\end{tikzpicture}\right), \mathrm{Rot}\left(\begin{tikzpicture}[baseline=-0.65ex,thick, scale=0.7]
\draw (0,-1.5) [->-] to (0,-0.5);
\draw (0,-0.5) [->-] to (0,0.5);
\draw (0,-0.5) [->] to (0,1.5);
\draw (0,-0.5) to [out=270, in=270] (-1,-0.2) [-<-] to [out=90, in=270] (-1,0.2) to [out=90, in=90] (0,0.5);
\draw (0.3, -1) node {$t$};
\draw (0.3, 1) node {$t$};
\draw (0.9, 0) node {$ts$};
\draw (-1.3,0) node {$s$};
\end{tikzpicture}\right)=s\,\,\mathrm{Rot} \left(\begin{tikzpicture}[baseline=-0.65ex,thick, scale=0.7]
\draw (0,-1.5) [->] to (0,1.5);
\draw (-0.25,0) node {$t$};
\draw (0.25,0) node {};
\end{tikzpicture}\right).$ 
\item $\mathrm{Rot}\left(\begin{tikzpicture}[baseline=-0.65ex,thick,scale=0.6]
\draw [->-] (0.5,-1.5) node[below]{$ts$}-- (0.5,-0.6);
\draw[->-] (0.5, -0.6) to [out=90, in=270] (0,0) to [out=90,in=270] (0.5,0.6) ;
\draw (0.5, -0.6) [->-] to [out=90, in=270] (1,0) to [out=90,in=270] (0.5,0.6);
\draw [->](0.5,0.6) -- (0.5,1.5) node[above]{$ts$};
\draw (-1, 0) node {$t$};
\draw (1.5, 0) node {$s$};
\end{tikzpicture}\right) =\mathrm{Rot}
\left(\begin{tikzpicture}[baseline=-0.65ex,thick,scale=0.6]
\draw [->](0,-1.5) -- (0, 1.5);
\draw (-0.5 , 0) node {$ts$};
\draw (0.25, 0) node {};
\end{tikzpicture}\right)$
\item $\mathrm{Rot}
\left (\begin{tikzpicture}[baseline=-0.65ex, thick]
\draw (0,-1) [->-] to [out=90,in=270] (0.5,-0.33);
\draw (0.5, -0.33) [->-] to [out=90,in=270] (0.5,0.33);
\draw (0.5, 0.33) [->] to [out=90,in=270] (0,1);
\draw (1,-1) [->-] to [out=90,in=270] (0.5,-0.33);
\draw (1,1) [<-] to [out=270,in=90] (0.5,0.33);
\draw (0, -1.25) node {$t$};
\draw (0.9,-1.25) node {$s$};
\draw (0,1.25) node {$t$};
\draw (1,1.25) node {$s$};
\draw (1,0) node {$ts$};
\end{tikzpicture}\right)=\mathrm{Rot} \left(\begin{tikzpicture}[baseline=-0.65ex, thick, scale=1.2]
\draw (0,-1)  [->]to (0,0.5);
\draw (1,-1)  [->]to (1,0.5);
\draw (0,0.75) node {$t$};
\draw (1,0.75) node {$s$};
\end{tikzpicture}\right)
$ 
\item $\mathrm{Rot} \left(\begin{tikzpicture}[baseline=-0.65ex, thick, scale=1.2]
\draw (0,-0.75)  [->]to (0,0.75) node[above]{$tr^{-1}$};
\draw (1,-0.75)  [->]to (1,0.75) node[above]{$sr$};
\draw (0,-0.25)  [->-]to (1,0.45);
\draw (0,-1) node {$t$};
\draw (1,-1) node {$s$};
\draw (0.5,-0.25) node {$r$};
\end{tikzpicture}\right)=\mathrm{Rot} \left(\begin{tikzpicture}[baseline=-0.65ex, thick, scale=1.2]
\draw (0,-0.75)  [->]to (0,0.75) node[above]{$tr^{-1}$};
\draw (1,-0.75)  [->]to (1,0.75) node[above]{$sr$};
\draw (1,-0.25)  [->-]to (0,0.45);
\draw (0,-1) node {$t$};
\draw (1,-1) node {$s$};
\draw (0.5,-0.25) node {$r^{-1}$};
\end{tikzpicture}\right)=\mathrm{Rot} \left(\begin{tikzpicture}[baseline=-0.65ex, thick, scale=1.2]
\draw (0,-0.75)  [->]to (1,0.75) node[above]{$sr$};
\draw (1,-0.75)  [->]to (0,0.75) node[above]{$tr^{-1}$};
\draw (0,-1) node {$t$};
\draw (1,-1) node {$s$};
\end{tikzpicture}\right)
$
\end{enumerate}
\end{prop}
\begin{proof}
The reason is obvious from the definition.
\end{proof}

\begin{prop}
\label{rotmove1}
The rotation number of a diagram is invariant under moves (II)--(V) in Fig.~\ref{fig:e25}, while its behavior under move I is described as follows.
\begin{eqnarray*}\label{r1}
&&\mathrm{Rot}\left( \begin{tikzpicture}[baseline, thick, scale=0.4]
\draw (1,-1)   -- (0.2,-0.2);
\draw (-1, -1) [->] -- (1, 1) node[above]{$t$};
\draw (-0.2,0.2) --  (-1,1) ;
\draw (-1, 1) arc (90:270:1);
\end{tikzpicture} \right )
=\mathrm{Rot}\left( \begin{tikzpicture}[baseline=-0.65ex, thick, scale=0.4]
\draw (-1,-1)   -- (-0.2,-0.2);
\draw (1, -1) -- (-1, 1) ;
\draw (0.2,0.2) [->] --  (1,1) node[above]{$t$};
\draw (-1, 1) arc (90:270:1);
\end{tikzpicture} \right )
=t\,\mathrm{Rot}\left( \begin{tikzpicture}[baseline=-0.65ex, thick, scale=0.4]
\draw (0, -1) [->] -- (0, 1) node[above]{$t$};
\end{tikzpicture} \right ), \\
&&\mathrm{Rot}\left(\begin{tikzpicture}[baseline=-0.65ex, thick, scale=0.4]
\draw (1,-1)   -- (0.2,-0.2);
\draw (-1, -1) -- (1, 1) ;
\draw (-0.2,0.2)  [->]--  (-1,1)node[above]{$t$} ;
\draw (1,-1) arc (-90:90:1);
\end{tikzpicture} \right )
= \mathrm{Rot}\left( \begin{tikzpicture}[baseline=-0.65ex, thick, scale=0.4]
\draw (-1,-1)   -- (-0.2,-0.2);
\draw (1, -1) [->] -- (-1, 1) node[above]{$t$};
\draw (0.2,0.2)  --  (1,1);
\draw (1,-1) arc (-90:90:1);
\end{tikzpicture} \right )
=t^{-1}\mathrm{Rot}\left( \begin{tikzpicture}[baseline=-0.65ex, thick, scale=0.4]
\draw (0, -1) [->] -- (0, 1) node[above]{$t$};
\end{tikzpicture} \right ).
\end{eqnarray*}
\end{prop}
\begin{proof}
The invariance under move (II) follows from (i), (v) and (vi) of Proposition~\ref{prop1}, as illustrated below.
\begin{eqnarray*}
\mathrm{Rot}\left(\begin{tikzpicture}[baseline=-0.65ex, thick, scale=0.4]
\draw (-1,-1.5)  [->] arc  (-90:90:2);
\draw (1.5,-1.5)[->] arc (-90:-270:2);
\draw (-2,-1.5) node {$s$};
\draw (2.5,-1.5) node {$t$};
\end{tikzpicture}\right ) =
\mathrm{Rot}\left(\begin{tikzpicture}[baseline=-0.65ex,thick,scale=0.6]
\draw [->-] (0.5,-1.2) -- (0.5,-0.6);
\draw[->-] (0.5, -0.6) to [out=90, in=270] (0,0) to [out=90,in=270] (0.5,0.6) ;
\draw (0.5, -0.6) [->-] to [out=90, in=270] (1,0) to [out=90,in=270] (0.5,0.6);
\draw (0.5,0.6) -- (0.5,1.2);
\draw (-0.7, 0) node {$t$};
\draw (1.5, 0) node {$s$};
\draw [->-] (0,-2)-- (0.5,-1.2);
\draw [->-] (1,-2)-- (0.5,-1.2);
\draw [<-] (1,2)-- (0.5,1.2);
\draw [<-] (0,2)-- (0.5,1.2);
\draw (-0.7, 2) node {$s$};
\draw (1.5, 2) node {$t$};
\draw (-0.7, -2) node {$s$};
\draw (1.5, -2) node {$t$};
\end{tikzpicture}\right)=
\mathrm{Rot} \left(\begin{tikzpicture}[baseline=-0.65ex, thick, scale=1.2]
\draw (0,-1)  [->]to (0,0.5);
\draw (1,-1)  [->]to (1,0.5);
\draw (0,0.75) node {$s$};
\draw (1,0.75) node {$t$};
\end{tikzpicture}\right).
\end{eqnarray*}

The invariance under move (III) can be proved as follows. The equalities follow from properties (i), (vii), (ii), (iii), (v) and (i) of Proposition~\ref{prop1} in order:
\begin{eqnarray*}
\mathrm{Rot} \left(\begin{tikzpicture}[baseline=-0.65ex, thick, scale=1]
\draw (-1,-1)  [->]to (1,1) node[above] {$r$};
\draw (1,-1)  [->]to (-1,1)node[above] {$t$}; 
\draw (0.5,-1)  [->]to (0.5,1)node[above] {$s$};
\end{tikzpicture}\right)=\mathrm{Rot} \left(\begin{tikzpicture}[baseline=-0.65ex, thick, scale=0.8]
\draw (-1,-2) -- (0,-0.5) -- (0, 0.5) [->]to(-1, 2);
\draw (1, -2)--(0.5, -1.5) -- (0.5, -1)--(1, -0.5) -- (1, 0.5)-- (0.5, 1)--(0.5, 1.5) [->]to (1,2);
\draw (0, -2)--(0.5, -1.5) -- (0.5, -1)--(0, -0.5) -- (0, 0.5)-- (0.5, 1)--(0.5, 1.5) [->]to (0,2);
\end{tikzpicture}\right)=
\mathrm{Rot} \left(\begin{tikzpicture}[baseline=-0.65ex, thick, scale=0.8]
\draw (-1,-2) -- (0.5, -1);
\draw (-1,2) [<-]-- (0.5, 1);
\draw (1, -2)--(0.5, -1.5) -- (0.5, -1)--(1, -0.5) -- (1, 0.5)-- (0.5, 1)--(0.5, 1.5) [->]to (1,2);
\draw (0, -2)--(0.5, -1.5) -- (0.5, -1)--(0, -0.5) -- (0, 0.5)-- (0.5, 1)--(0.5, 1.5) [->]to (0,2);
\end{tikzpicture}\right)
=\mathrm{Rot} \left(\begin{tikzpicture}[baseline=-0.65ex, thick, scale=0.8]
\draw (-1,-2) -- (0.5, -1.5);
\draw (-1,2) [<-]-- (0.5, 1.5);
\draw (1, -2)--(0.5, -1.5) -- (0.5, -1)--(1, -0.5) -- (1, 0.5)-- (0.5, 1)--(0.5, 1.5) [->]to (1,2);
\draw (0, -2)--(0.5, -1.5) -- (0.5, -1)--(0, -0.5) -- (0, 0.5)-- (0.5, 1)--(0.5, 1.5) [->]to (0,2);
\end{tikzpicture}\right)\\
=\mathrm{Rot} \left(\begin{tikzpicture}[baseline=-0.65ex, thick, scale=0.8]
\draw (-1,-2) -- (0.5, -1);
\draw (-1,2) [<-]-- (0.5, 1);
\draw (1, -2)--(0.5, -1) --(0.5, 1) [->]to (1,2);
\draw (0, -2)--(0.5, -1) --(0.5, 1) [->]to (0,2);
\end{tikzpicture}\right)=\mathrm{Rot} \left(\begin{tikzpicture}[baseline=-0.65ex, thick, scale=0.6]
\draw (-1,-2) -- (0, 0);
\draw (0, 0) [->]--(-1,2)  node[above]{$t$};
\draw (1, -2)--(0, 0) [->]to (1,2) node[above]{$r$};
\draw (0, -2)--(0, 0) [->]to (0,2) node[above]{$s$};
\end{tikzpicture}\right).
\end{eqnarray*}

Similarly, for the other configuration:
\begin{eqnarray*}
\mathrm{Rot} \left(\begin{tikzpicture}[baseline=-0.65ex, thick, scale=1]
\draw (-1,-1)  [->]to (1,1) node[above] {$r$};
\draw (1,-1)  [->]to (-1,1)node[above] {$t$}; 
\draw (-0.5,-1)  [->]to (-0.5,1)node[above] {$s$};
\end{tikzpicture}\right)=
\mathrm{Rot} \left(\begin{tikzpicture}[baseline=-0.65ex, thick, scale=0.6]
\draw (-1,-2) -- (0, 0);
\draw (0, 0) [->]--(-1,2)  node[above]{$t$};
\draw (1, -2)--(0, 0) [->]to (1,2) node[above]{$r$};
\draw (0, -2)--(0, 0) [->]to (0,2) node[above]{$s$};
\end{tikzpicture}\right)
\end{eqnarray*}
Since both sides of move (III) equal the same rotation number, the invariance follows.

The invariance under move (IV) can be established in the same manner, and we omit the detailed proof. The behavior under move (I) follows from (i) and (iv) of Proposition~\ref{prop1}.
\end{proof}

\begin{rem}
Two graph diagrams are regularly homotopic if they are connected by a finite sequence of Reidemeister moves (II), (III) and (IV) in Fig. \ref{fig:e25}. Proposition \ref{rotmove1} shows that the rotation number $\mathrm{Rot}(D)$ is a regular homotopy invariant for transverse graph diagrams. We note that Nikkuni \cite{MR2607409} showed the Wu invariant \cite{MR215305, MR124061} is a complete regular homotopy invariant for graphs, suggesting a potential relationship with our rotation number that merits further investigation. 
\end{rem}

\medskip

\subsection{Relations with existing invariants.}
We now establish connections between our rotation number and several existing invariants.

For an oriented link $L$ in $S^3$, there is a natural homomorphism $\phi: H_1(S^3 \setminus L; \mathbb{Z}) \to \mathbb{Z}$ that sends each oriented meridian of $L$ to the generator $1$. If $D$ is a diagram of $L$, then under this homomorphism, our rotation number $\mathrm{Rot}(D)$, as defined in Definition~\ref{rotgraph}, coincides with the classical rotation number $w(D)$.

\begin{prop}
\label{curve}
Let $D$ be a connected diagram of an oriented link. Then
$$\phi (\mathrm{Rot}(D))=\omega(D).$$
\end{prop}
\begin{proof}
This is evident from Viro's formula in Theorem \ref{Viroformula}.
\end{proof}


More generally, for any graph $G$, let $E$ denote its edge set. Given an abelian group $M$, an \textit{$M$-coloring} of $G$ is a map 
$c: E\to M$ such that for each vertex $v$,
$$\sum_{\text{$e$: pointing into $v$}} c(e)=\sum_{\text{$e$: pointing out of $v$}} c(e).$$
Such a coloring induces a homomorphism $\phi_c: H_1(S^3\setminus G; \mathbb{Z})\to M$ that sends the oriented meridian $t_e$ of each edge $e$ to $c(e)$.

When $G$ is an oriented transverse graph without sinks or sources and $c$ takes positive integer values, the pair $(G, c)$ is called an \textit{MOY graph}. These graphs, central to the MOY calculus \cite{MR1659228}, will be essential in relating our invariant to Viro's polynomial later in this paper. Note that while general colorings use arbitrary abelian groups $M$, MOY graphs specifically require $M = \mathbb{Z}$ with positive integer colors.

Given an MOY graph diagram $(D, c)$, we construct an oriented link diagram $L_{(D, c)}$ by the local transformation illustrated in Fig. \ref{fig:e24}, where we replace each edge with color $i$ by $i$ parallel strands. 

\begin{figure}[h!]
\begin{tikzpicture}[baseline=-0.65ex, thick, scale=0.8]
\draw (0,-1)  [->]-- (0,1);
\draw (0.3, 0.8) node {$4$} ;
\draw (1, 0) node {$\Longrightarrow$} ;
\draw (2,-1)  [->]-- (2,1);
\draw (2.5,-1)  [->]-- (2.5,1);
\draw (3,-1)  [->]-- (3,1);
\draw (3.5,-1)  [->]-- (3.5,1);
\end{tikzpicture}\quad \quad\quad
\begin{tikzpicture}[baseline=-0.65ex, thick, scale=0.8]
\draw (0,-1)  [->-]-- (0,0);
\draw (0,0)  [->]-- (1,1);
\draw (0,0)  [->]-- (-1,1);
\draw (1.3, 0.8) node {$3$} ;
\draw (-1.3, 0.8) node {$1$} ;
\draw (0.3, -0.8) node {$4$} ;
\end{tikzpicture}$\Longrightarrow$
\begin{tikzpicture}[baseline=-0.65ex, thick, scale=0.8]
\draw (1.2,-1) to [out=90, in=270] (1.2,0) [->] to [out=90,in=225] (2,1);
\draw (0.7,-1) to [out=90, in=270] (0.7,0) [->] to [out=90,in=225] (1.5,1);
\draw (-0.3,-1) to [out=90, in=270] (-0.3,0) [->] to [out=90,in=315] (-1.1,1);
\draw (0.2,-1) to [out=90, in=270] (0.2,0) [->] to [out=90,in=225] (1,1);
\end{tikzpicture}
	\caption{Constructing the link $L_{(D, c)}$ from an MOY  graph diagram $(D, c)$.}
	\label{fig:e24}
\end{figure}
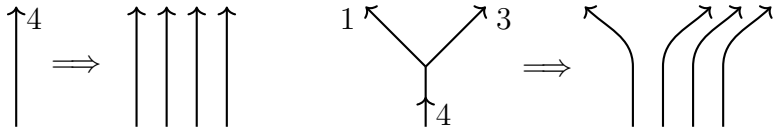

The following proposition relates our rotation number to the classical rotation number of the associated link:

\begin{prop}
\label{rotrelation}
For an MOY graph diagram $(D, c)$, we have
$$\phi_c(\mathrm{Rot}(D))=w(L_{(D, c)}).$$
\end{prop}
\begin{proof}
The transformation from $(D, c)$ to $L_{(D,c)}$ is achieved through moves from Proposition~\ref{prop1}: splitting edges via (v), contracting edges via (ii) and (iii), and resolving vertices via (vi). Since these moves preserve rotation number, we have $\mathrm{Rot}(D) = \mathrm{Rot}(L_{(D,c)})$. Then the result follows from Proposition~\ref{curve}.
\end{proof}

\begin{rem}
This result connects our rotation number to the \textit{curliness} $\mathcal{C}(D, c) = t^{w(L_{(D, c)})/2}$ defined in \cite{MR4090586}. Specifically, we have $\mathcal{C}(D, c) = t^{\phi_c(\mathrm{Rot}(D))/2}$, showing that our rotation number generalizes the exponent in the curliness invariant.
\end{rem}

\section{A multi-variable Alexander polynomial}
\subsection{Kauffman states}
We recall the definition of Kauffman state from \cite[Definition 2.4]{MR4090586}.
Suppose $D$ is a connected diagram of a transverse graph. We can obtain a {\it decorated diagram} $(D,\delta)$ by placing a base point $\delta$ on an edge of $D$ and drawing a circle around each vertex of $D$.  Then we define

\begin{enumerate}
\item  $\operatorname{Cr}(D)$: the set of crossings, including crossings of types \diaCrossP and \diaCrossN which are double points of the diagram and type \diaCircle which are intersection points around each vertex between incoming edges and the circle.\\
\item $\operatorname{Re}(D)$: the set of regions, including the {\it regular regions} of $\mathbb{R}^{2}$ separated by $D$ and the {\it circle regions} around the vertices. The regions adjacent to the base point $\delta$ are called {\it marked regions}; the others are {\it unmarked regions}.\\

\item Corners: For a crossing of type \diaCrossP or \diaCrossN, there are four corners: {\it north}, {\it south}, {\it west}, and {\it east}. For a crossing of type \diaCircle, there are three corners: the one inside the circle region is the {\it north} corner, the one on the left is the {\it west} corner and the one on the right is the {\it east} corner.  Every corner belongs to a unique region in $\operatorname{Re}(D)$.  \\
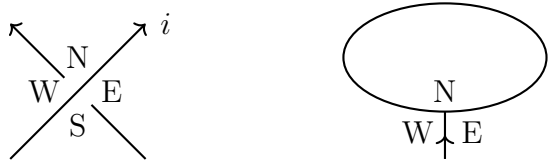
\begin{figure}[h!]
\begin{tikzpicture}[baseline=-0.65ex, thick, scale=0.9]
\draw (-1,-1) [->] to (1,1);
\draw (1.3, 1) node {$i$};
\draw (1,-1) -- (0.2,-0.2);
\draw (-0.2,0.2) [->] to (-1,1);
\draw (0, 0.5) node {N};
\draw (0, -0.5) node {S};
\draw (0.5, 0) node {E};
\draw (-0.5, 0) node {W};
\end{tikzpicture}\hspace{2cm}
\begin{tikzpicture}[baseline=-0.65ex, thick, scale=0.9]
\draw (0, 0.5) ellipse (1.5cm and 0.8cm);
\draw (0,-1) [->-] to (0,-0.3);
\draw (-0.4, -0.6) node {W};
\draw (0.4, -0.6) node {E};
\draw (0, 0) node {N};
\end{tikzpicture}
\caption{Corner labels for crossings of type \diaCrossP (left) and \diaCircle (right).}
\label{fig:corners}
\end{figure}

\end{enumerate}

There are two marked regions, denoted by $R_u$ and $R_v$.  
The assumption that each edge of $G$ has a non-trivial meridian in homology ensures that $R_u \neq R_v$.
Moreover, if $D$ is connected (as a diagram in $\mathbb{R}^2$), then
\[
\lvert \operatorname{Re}(D) \rvert = \lvert \operatorname{Cr}(D) \rvert + 2.
\]

A {\it Kauffman state}, or simply a {\it state}, of a decorated diagram $(D, \delta)$ is a bijection
\[
s: \operatorname{Cr}(D) \longrightarrow \operatorname{Re}(D) \setminus \{R_u, R_v\},
\]
which assigns to each crossing in $\operatorname{Cr}(D)$ one of its adjacent corners.  Since each corner belongs to a unique region, this is equivalent to assigning to each crossing an adjacent region via the specific corner. 
We denote by $S(D,\delta)$ the set of all such states.

As an illustration, Fig.~\ref{fig:e5state1} exhibits a decorated diagram and one of its Kauffman states.

\begin{figure}[h!]
\begin{tikzpicture}[baseline=-0.65ex, thick, scale=1.3]
\draw (0,-1)  to [out=90,in=270] (0.5,-0.33);
\draw (0,-1) to [out=270,in=180] (1.5,-2);
\draw (1.5,-2) to [out=0,in=0] (1.5,1);
\draw (0.5, -0.33) [->-] to [out=90,in=270] (0.5,0.33);
\draw (0.5, 0.33) [->] to [out=90,in=0] (-0.5,1);
\draw (1,-1)  to [out=90,in=270] (0.5,-0.33);
\draw (1,-1) to [out=270,in=45] (0.6,-1.6);
\draw (0.3,-1.8) to [out=225,in=0] (-0.5,-2);
\draw (-0.5,-2) to [out=180,in=180] (-0.5,1);
\draw (1.5,1) [<-] to [out=180,in=90] (0.5,0.33);
\draw (0.5, -0.33) node[circle,fill,inner sep=1pt]{};
\draw (0.5, -0.33) circle (0.3);
\draw (0.5, 0.33) node[circle,fill,inner sep=1pt]{};
\draw (0.5, 0.33) circle (0.3);
\draw (-1.35,0) node {$*$};
\draw (-1,0) node {$\delta$};
\draw (-0.5,-0.5) node {$\star$};
\draw (-1.7,-0.5) node {$\star$};
\draw (0.8,-0.6) node {$\bullet$};
\draw (0.5,0.15) node {$\bullet$};
\draw (0.45,-0.5) node {$\bullet$};
\draw (0.45,-1.5) node {$\bullet$};
\end{tikzpicture}
	\caption{A decorated diagram with four regular regions, two circle regions, one crossing of type \diaCrossN, and three crossings of type \diaCircle. The marked regions adjacent to $\delta$ are indicated by $\star$, and a Kauffman state is indicated by $\bullet$'s.}
	\label{fig:e5state1}
\end{figure}
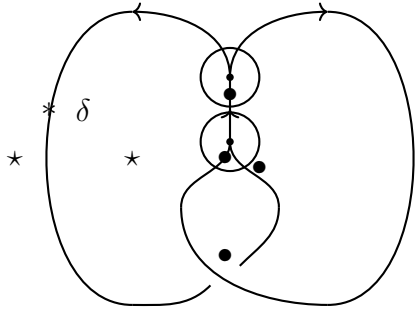

\medskip

\subsection{Kauffman state sum}
The aim here is to define a value $\langle D\rangle$, which is the multi-variable version of the Kauffman state sum defined in \cite{MR4090586}. 

\begin{defn}
\rm
Suppose the base point $\delta$ lies on an edge with oriented meridian $t$, and the winding numbers of the adjacent marked regions are $x$ and $xt$, respectively. Define
\begin{equation*}\label{delta}
\vert \delta \vert=x-xt.
\end{equation*}
Note that since all edges have non-trivial homology, we have $\vert \delta \vert \neq 0 \in \mathbb{Z}H_1(S^3\setminus $G$; \mathbb{Z})$.
\end{defn}

\begin{defn}
\rm
Suppose $(D, \delta)$ is a connected decorated diagram with $N$ crossings $C_1, C_2, \cdots, C_N$ in $\operatorname{Cr}(D)$ and $N+2$ regions $R_1, R_2, \cdots, R_{N+2}$ in $\operatorname{Re}(D)$. 
\begin{enumerate}
\item Define the local contributions $M_{C_p}^{\triangle}$ and $A_{C_p}^{\triangle}$ associated to each corner $\triangle$ around the crossing $C_p$ as shown in Fig. ~\ref{fig:e1}. 
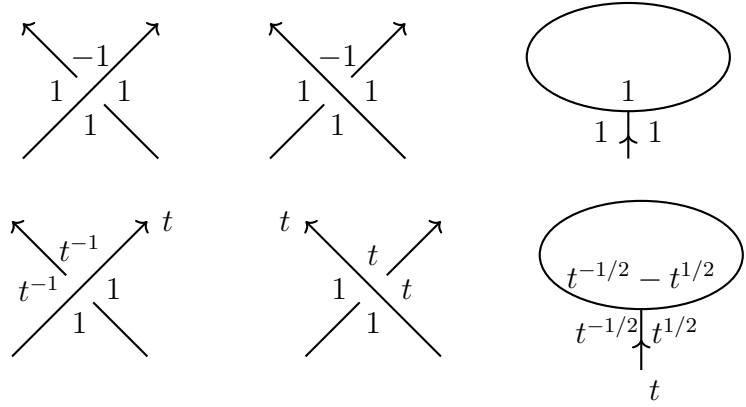
\begin{figure}[h!]
\begin{tikzpicture}[baseline=-0.65ex, thick, scale=0.9]
\draw (-1,-1) [->] to (1,1);
\draw (1,-1) -- (0.2,-0.2);
\draw (-0.2,0.2) [->] to (-1,1);
\draw (0, 0.5) node {$-1$};
\draw (0, -0.5) node {$1$};
\draw (0.5, 0) node {$1$};
\draw (-0.5, 0) node {$1$};
\end{tikzpicture}\hspace{1.3cm}
\begin{tikzpicture}[baseline=-0.65ex, thick, scale=0.9]
\draw (1,-1) [->] to (-1,1);
\draw (-1,-1) -- (-0.2,-0.2);
\draw (0.2,0.2) [->] to (1,1);
\draw (0, 0.5) node {$-1$};
\draw (0, -0.5) node {$1$};
\draw (0.5, 0) node {$1$};
\draw (-0.5, 0) node {$1$};
\end{tikzpicture}
\hspace{1.3cm}
\begin{tikzpicture}[baseline=-0.65ex, thick, scale=0.9]
\draw (0, 0.5) ellipse (1.5cm and 0.8cm);
\draw (0,-1) [->-] to (0,-0.3);
\draw (-0.4, -0.6) node {$1$};
\draw (0.4, -0.6) node {$1$};
\draw (0, 0) node {$1$};
\end{tikzpicture}

\vspace{5mm}

\begin{tikzpicture}[baseline=-0.65ex, thick, scale=0.9]
\draw (-1,-1) [->] to (1,1);
\draw (1.3, 1) node {$t$};
\draw (1,-1) -- (0.2,-0.2);
\draw (-0.2,0.2) [->] to (-1,1);
\draw (0, 0.6) node {$t^{-1}$};
\draw (0, -0.5) node {$1$};
\draw (0.5, 0) node {$1$};
\draw (-0.6, 0) node {$t^{-1}$};
\end{tikzpicture}\hspace{1cm}
\begin{tikzpicture}[baseline=-0.65ex, thick, scale=0.9]
\draw (1,-1) [->] to (-1,1);
\draw (-1.3, 1) node {$t$};
\draw (-1,-1) -- (-0.2,-0.2);
\draw (0.2,0.2) [->] to (1,1);
\draw (0, 0.5) node {$t$};
\draw (0, -0.5) node {$1$};
\draw (0.5, 0) node {$t$};
\draw (-0.5, 0) node {$1$};
\end{tikzpicture}
\hspace{1cm}
\begin{tikzpicture}[baseline=-0.65ex, thick, scale=0.9]
\draw (0, 0.5) ellipse (1.5cm and 0.8cm);
\draw (0,-1.2) [->-] to (0,-0.3);
\draw (-0.5, -0.6) node {$t^{-1/2}$};
\draw (0.5, -0.6) node {$t^{1/2}$};
\draw (0, 0.2) node {$t^{-1/2}-t^{1/2}$};
\draw (0.2, -1.5) node {$t$};
\end{tikzpicture}
	\caption{Local contributions $M_{C_p}^{\triangle}$ (top) and $A_{C_p}^{\triangle}$ (bottom). Here $t$ denotes the oriented meridian of the nearby edge.}
	\label{fig:e1}
\end{figure}

\item For each state $s\in S(D, \delta)$, define 
$$M(s) :=\prod_{p=1}^{N}M_{C_p}^{s(C_{p})}, \quad
A(s) :=\prod_{p=1}^{N}A_{C_p}^{s(C_{p})}. $$

\item The {\it state sum} is defined as \begin{equation}
\label{alexander}
\langle D\rangle_\delta:=\vert \delta\vert^{-1}\sum_{s\in S(D, \delta)} M(s)\cdot A(s).
\end{equation}
If $S(D, \delta)=\emptyset$, we set $\langle D\rangle =0$. In particular, if the diagram $D$ is disconnected, then $S(D, \delta)=\emptyset$ and $\langle D\rangle =0$.
\end{enumerate}

\end{defn}

\begin{lemma}
\label{skein}
For a choice of the base point $\delta$ which is placed in the same position on both sides, the state sum satisfies the following skein relations. 
\begin{align}
&\left<\begin{tikzpicture}[baseline=-0.65ex, thick, scale=0.5]
\draw (1,-1)   -- (0.2,-0.2);
\draw (-1, -1) [->] -- (1, 1) node[above]{$t$};
\draw (-0.2,0.2) [->] to (-1,1)  node[above]{$s$};
\end{tikzpicture} \right>
  =\frac{-t^{-\frac{1}{2}}s^{-\frac{1}{2}}}{\{\frac{1}{2}\}_t \{\frac{1}{2}\}_{s}}\cdot
\left<
\begin{tikzpicture}[baseline=-0.65ex, thick, scale=0.5]
\draw (0,-2) node[below]{$t$} [->-] to (0, 0);
\draw (0,0) --  (0, 1);
\draw (0, 1) [->] to (0,2) node[above]{$s$};
\draw (0,1) node[circle,fill,inner sep=1pt]{};
\draw (2,0) node[circle,fill,inner sep=1pt]{};
\draw (2,-2) [->-] node[below]{$s$} to (2,0);
\draw (2,0) [->] to (2,2) node[above]{$t$};
\draw (2,-0) [->-] to (0, 1);
\draw (1, -0.3) node {$st^{-1}$};
\end{tikzpicture}\right>
+ \, \frac{s^{-\frac{1}{2}}}{\{\frac{1}{2}\}_t \{\frac{1}{2}\}_{ts}}\cdot
\left<\begin{tikzpicture}[baseline=-0.65ex, thick]
\draw (0,-1) [->-] to  (0.5,-0.33);
\draw (0.5, -0.33) [->-] to  (0.5,0.33);
\draw (0.5, 0.33) [->] to  (0,1);
\draw (1,-1) [->-] to (0.5,-0.33);
\draw (1,1) [<-] to  (0.5,0.33);
\draw (0, -1.25) node {$t$};
\draw (0.9,-1.25) node {$s$};
\draw (0,1.25) node {$s$};
\draw (1,1.25) node {$t$};
\draw (1,0) node {$ts$};
\draw (0.5,0.33) node[circle,fill,inner sep=1pt]{};
\draw (0.5,-0.33) node[circle,fill,inner sep=1pt]{};
\end{tikzpicture}\right>,\\
&\left<
\begin{tikzpicture}[baseline=-0.65ex, thick, scale=0.5]
\draw (1,-1) [->]  -- (-1,1)  node[above]{$s$};
\draw (-1, -1) to (-0.2, -0.2);
\draw (0.2, 0.2) [->] -- (1, 1) node[above]{$t$};
\end{tikzpicture}\right>
  =\frac{-t^{\frac{1}{2}}s^{\frac{1}{2}}}{\{\frac{1}{2}\}_t \{\frac{1}{2}\}_{s}}\cdot
\left<
\begin{tikzpicture}[baseline=-0.65ex, thick, scale=0.5]
\draw (0,-2) node[below]{$t$} [->-] to (0, 0);
\draw (0,0) --  (0, 1);
\draw (0, 1) [->] to (0,2) node[above]{$s$};
\draw (0,1) node[circle,fill,inner sep=1pt]{};
\draw (2,0) node[circle,fill,inner sep=1pt]{};
\draw (2,-2) [->-] node[below]{$s$} to (2,0);
\draw (2,0) [->] to (2,2) node[above]{$t$};
\draw (2,-0) [->-] to (0, 1);
\draw (1, -0.3) node {$st^{-1}$};
\end{tikzpicture}\right>
+ \, \frac{s^{\frac{1}{2}}}{\{\frac{1}{2}\}_t \{\frac{1}{2}\}_{ts}}\cdot
\left<\begin{tikzpicture}[baseline=-0.65ex, thick]
\draw (0,-1) [->-] to  (0.5,-0.33);
\draw (0.5, -0.33) [->-] to  (0.5,0.33);
\draw (0.5, 0.33) [->] to  (0,1);
\draw (1,-1) [->-] to (0.5,-0.33);
\draw (1,1) [<-] to  (0.5,0.33);
\draw (0, -1.25) node {$t$};
\draw (0.9,-1.25) node {$s$};
\draw (0,1.25) node {$s$};
\draw (1,1.25) node {$t$};
\draw (1,0) node {$ts$};
\draw (0.5,0.33) node[circle,fill,inner sep=1pt]{};
\draw (0.5,-0.33) node[circle,fill,inner sep=1pt]{};
\end{tikzpicture}\right>,
\end{align}
where $\{\frac{1}{2}\}_t:=t^{\frac{1}{2}}-t^{-\frac{1}{2}}$.
\end{lemma}
\begin{proof}
These two relations can be proved by comparing the Kauffman states and their contributions to $\langle D \rangle$, using the same method as in \cite[Theorem 4.1]{MR4090586}.
\end{proof}

\begin{rem}
The state sum $\langle D\rangle_\delta$ depends a priori on the choice of base point $\delta$. However, we will prove in Proposition~\ref{initial} that it is independent of this choice. Therefore, in subsequent discussions we will use the simplified notation $\langle D\rangle$, with the understanding that the base point independence has been established.
\end{rem}

\begin{prop}
\label{invtheo}
The state sum $\langle D \rangle $ is invariant under the Reidemeister moves (II) -- (V)  in Fig. \ref{fig:e25}, and its behavior under Reidemeister move (I) is given by
\begin{equation*}\label{r1}
t \left< \begin{tikzpicture}[baseline, thick, scale=0.4]
\draw (1,-1)   -- (0.2,-0.2);
\draw (-1, -1) [->] -- (1, 1) node[above]{$t$};
\draw (-0.2,0.2) --  (-1,1) ;
\draw (-1, 1) arc (90:270:1);
\end{tikzpicture} \right >
=\left< \begin{tikzpicture}[baseline=-0.65ex, thick, scale=0.4]
\draw (1,-1)   -- (0.2,-0.2);
\draw (-1, -1) -- (1, 1) ;
\draw (-0.2,0.2)  [->]--  (-1,1)node[above]{$t$} ;
\draw (1,-1) arc (-90:90:1);
\end{tikzpicture} \right >
=\left< \begin{tikzpicture}[baseline=-0.65ex, thick, scale=0.4]
\draw (-1,-1)   -- (-0.2,-0.2);
\draw (1, -1) -- (-1, 1) ;
\draw (0.2,0.2) [->] --  (1,1) node[above]{$t$};
\draw (-1, 1) arc (90:270:1);
\end{tikzpicture} \right >
=t^{-1} \left< \begin{tikzpicture}[baseline=-0.65ex, thick, scale=0.4]
\draw (-1,-1)   -- (-0.2,-0.2);
\draw (1, -1) [->] -- (-1, 1) node[above]{$t$};
\draw (0.2,0.2)  --  (1,1);
\draw (1,-1) arc (-90:90:1);
\end{tikzpicture} \right >
=\left< \begin{tikzpicture}[baseline=-0.65ex, thick, scale=0.4]
\draw (0, -1) [->] -- (0, 1) node[above]{$t$};
\end{tikzpicture} \right >.
\end{equation*}

\end{prop}
\begin{proof}
The invariance under moves (II)--(V) and the change under move (I) follow from the same state-sum analysis as in \cite[Proposition 3.1]{MR4090586}.
\end{proof} 

\medskip

\subsection{A normalized multi-variable Alexander polynomial}
In this section, we use the rotation number $\mathrm{Rot}(D)$ defined in Section 2 to normalize the state sum $\langle D \rangle$ into a topological invariant for framed transverse graphs. This approach generalizes the normalization for MOY graphs given in \cite{MR4090586}. 

Let $G$ be a transverse graph.  
Recall that a {\it framing} of $G$ is an embedded compact surface $F \subset S^3$ such that $G$ is a deformation retract of $F$.  
A {\it framed graph} is a graph equipped with such a framing.  
More concretely, each vertex of $G$ is replaced by a disk in $F$, with the vertex at the center, and each edge of $G$ is replaced by a strip $[0,1] \times [0,1]$, where the sides $[0,1] \times \{0,1\}$ are attached to the boundaries of the adjacent vertex disks, and the edge itself corresponds to $\{\tfrac{1}{2}\} \times [0,1]$.  
It is straightforward to see that a framed transverse graph is equivalent to a ribbon graph in \cite{MR1036112}, where vertices are coupons and edges are annuli.

Every graph diagram of $G$ in $\mathbb{R}^{2}$ carries a natural {\it blackboard framing}, whose projection in $\mathbb{R}^{2}$ is given by the tubular neighborhood of the diagram in the plane.  
From now on, we represent framed transverse graphs by their graph diagrams with blackboard framing.  
For framed transverse graphs, we have the following well-known fact.


\begin{lemma}
Two graph diagrams represent the same framed transverse graph if and only if they are related by a sequence of Reidemeister moves as shown in Fig.~\ref{fig:framed}.
\end{lemma}

\begin{figure}[h!]
I': \quad \begin{tikzpicture}[baseline, thick, scale=0.4]
\draw (1,-3)   -- (0.2,-2.2);
\draw (-1, -3) -- (1, -1);
\draw (-0.2,-1.7) --  (-1,-1) ;
\draw (-1, -1) arc (90:270:1);
\draw (1, -1) -- (1, 1) ;
\draw (-1,1)   -- (-0.2,1.7);
\draw (1, 1) -- (-1, 3) ;
\draw (0.2,2.2)  --  (1,3);
\draw (-1, 3) arc (90:270:1);
\end{tikzpicture}\quad  $\longleftrightarrow$ \quad
\begin{tikzpicture}[baseline=-0.65ex, thick, scale=0.4]
\draw (0, -1) -- (0, 1) ;
\end{tikzpicture}\quad $\longleftrightarrow$ \quad
\begin{tikzpicture}[baseline, thick, scale=0.4]
\draw (1,-3)   -- (0.2,-2.2);
\draw (1, 1) arc (-90:90:1);
\draw (-1, -3) -- (1, -1);
\draw (-0.2,-1.7) --  (-1,-1) ;
\draw (-1, -1) -- (-1, 1) ;
\draw (-1,1)   -- (-0.2,1.7);
\draw (1, 1) -- (-1, 3) ;
\draw (0.2,2.2)  --  (1,3);
\draw (1, -3) arc (-90:90:1);
\end{tikzpicture}
\\  \vspace{3mm}
II: \quad \begin{tikzpicture}[baseline, thick, scale=0.4]
\draw (-1, 2) arc (90:270:2);
\draw (-3.5, 2) arc (90:60:2);
\draw (-1.5, 0) arc (0:40:2);
\draw (-1.5, 0) arc (0:-40:2);
\draw (-3.5, -2) arc (270:300:2);
\end{tikzpicture} \quad  $\longleftrightarrow$ \quad
\begin{tikzpicture}[baseline, thick, scale=0.4]
\draw (-1, 2) arc (110:250:2);
\draw (-5, 2) arc (70:-70:2);
\end{tikzpicture} \\  \vspace{3mm}
III: \quad\begin{tikzpicture}[baseline, thick, scale=0.4]
\draw (-2, 2) -- (-0.4, 0.4) ;
\draw (-2, 1) -- (-1.5, 1) ;
\draw (1.5, 1) -- (2, 1) ;
\draw (-0.5, 1) -- (0.5, 1) ;
\draw (0.4, -0.4) -- (2, -2) ;
\draw (2, 2) -- (-2, -2) ;
\end{tikzpicture}\quad  $\longleftrightarrow$ \quad
\begin{tikzpicture}[baseline, thick, scale=0.4]
\draw (-2, 2) -- (-0.4, 0.4) ;
\draw (-2, -1) -- (-1.5, -1) ;
\draw (1.5, -1) -- (2, -1) ;
\draw (-0.5, -1) -- (0.5, -1) ;
\draw (0.4, -0.4) -- (2, -2) ;
\draw (2, 2) -- (-2, -2) ;
\end{tikzpicture}\quad\quad\quad
\begin{tikzpicture}[baseline, thick, scale=0.4]
\draw (-2, 2) -- (2, -2) ;
\draw (-2, 1) -- (-1.5, 1) ;
\draw (1.5, 1) -- (2, 1) ;
\draw (-0.5, 1) -- (0.5, 1) ;
\draw (0.4, -0.4) -- (2, -2) ;
\draw (2, 2) -- (0.4, 0.4) ;
\draw (-2, -2) -- (-0.4, -0.4) ;
\end{tikzpicture}\quad  $\longleftrightarrow$ \quad
\begin{tikzpicture}[baseline, thick, scale=0.4]
\draw (-2, 2) -- (2, -2) ;
\draw (-2, -1) -- (-1.5, -1) ;
\draw (1.5, -1) -- (2, -1) ;
\draw (-0.5, -1) -- (0.5, -1) ;
\draw (2, 2) -- (0.4, 0.4) ;
\draw (-2, -2) -- (-0.4, -0.4) ;
\end{tikzpicture}\\\vspace{4mm}
IV: \quad \begin{tikzpicture}[baseline=-0.65ex, thick, scale=1]
\draw (-1, -1) -- (0, 0);
\draw (-0.5, -1) -- (0, 0);
\draw (0, 0) -- (1, 1);
\draw (1, -1) -- (0, 0);
\draw (0, 0) -- (-1,1);
\draw (0, 0) -- (0.5, 1);
\draw (0, 0) node[circle,fill,inner sep=1.5pt]{};
\draw [dashed] (-0.7, 0)--(0.7, 0);
\draw (0.1, -0.5) node{$...$};
\draw (-0.1, 0.5) node{$...$};
\draw (-1.2, 0.7) -- (-0.9, 0.7);
\draw (0.9, 0.7) -- (1.2, 0.7);
\draw (0.65, 0.7) -- (0.45, 0.7);
\draw (-0.6, 0.7) -- (0.2, 0.7);
\end{tikzpicture}  $\longleftrightarrow$
\begin{tikzpicture}[baseline=-0.65ex, thick, scale=1]
\draw (-1, -1) --  (0, 0);
\draw (-0.5, -1) -- (0, 0);
\draw (0, 0) -- (1, 1);
\draw (1, -1)--  (0, 0);
\draw (0, 0) --  (-1,1);
\draw (0, 0) --  (0.5, 1);
\draw (0, 0) node[circle,fill,inner sep=1.5pt]{};
\draw [dashed] (-0.7, 0)--(0.7, 0);
\draw (0.1, -0.5) node{$...$};
\draw (-0.1, 0.5) node{$...$};
\draw (-1.2, -0.7) -- (-0.9, -0.7);
\draw (0.9, -0.7) -- (1.2, -0.7);
\draw (-0.65, -0.7) -- (-0.45, -0.7);
\draw (0.6, -0.7) -- (-0.2, -0.7);
\end{tikzpicture}\quad\quad
\begin{tikzpicture}[baseline=-0.65ex, thick, scale=1]
\draw (-1, -1) --  (0, 0);
\draw (-0.5, -1) -- (0, 0);
\draw (0, 0) -- (0.6, 0.6);
\draw (0.8, 0.8) --  (1, 1);
\draw (1, -1) --  (0, 0);
\draw (0, 0) -- (-0.6,0.6);
\draw (-0.8, 0.8) --  (-1,1);
\draw (0, 0) -- (0.3, 0.6);
\draw (0.4, 0.8) --  (0.5, 1);
\draw (0, 0) node[circle,fill,inner sep=1.5pt]{};
\draw [dashed] (-0.7, 0)--(0.7, 0);
\draw (0.1, -0.5) node{$...$};
\draw (-0.1, 0.5) node{$...$};
\draw (-1.2, 0.7) -- (1.2, 0.7);
\end{tikzpicture}  $\longleftrightarrow$
\begin{tikzpicture}[baseline=-0.65ex, thick, scale=1]
\draw (-1, -1) -- (-0.8, -0.8);
\draw (-0.6, -0.6) --  (0, 0);
\draw (-0.5, -1) -- (-0.4, -0.8);
\draw (-0.3, -0.6) --  (0, 0);
\draw (0, 0) -- (1, 1);
\draw (0.6, -0.6) -- (0, 0);
\draw (1, -1) -- (0.8, -0.8);
\draw (0, 0) -- (-1,1);
\draw (0, 0) -- (0.5, 1);
\draw (0, 0) node[circle,fill,inner sep=1.5pt]{};
\draw [dashed] (-0.7, 0)--(0.7, 0);
\draw (0.1, -0.5) node{$...$};
\draw (-0.1, 0.5) node{$...$};
\draw (-1.2, -0.7) -- (1.2, -0.7);
\end{tikzpicture}
\caption{Reidemeister moves for framed transverse graph diagrams.}
\label{fig:framed}
\end{figure}
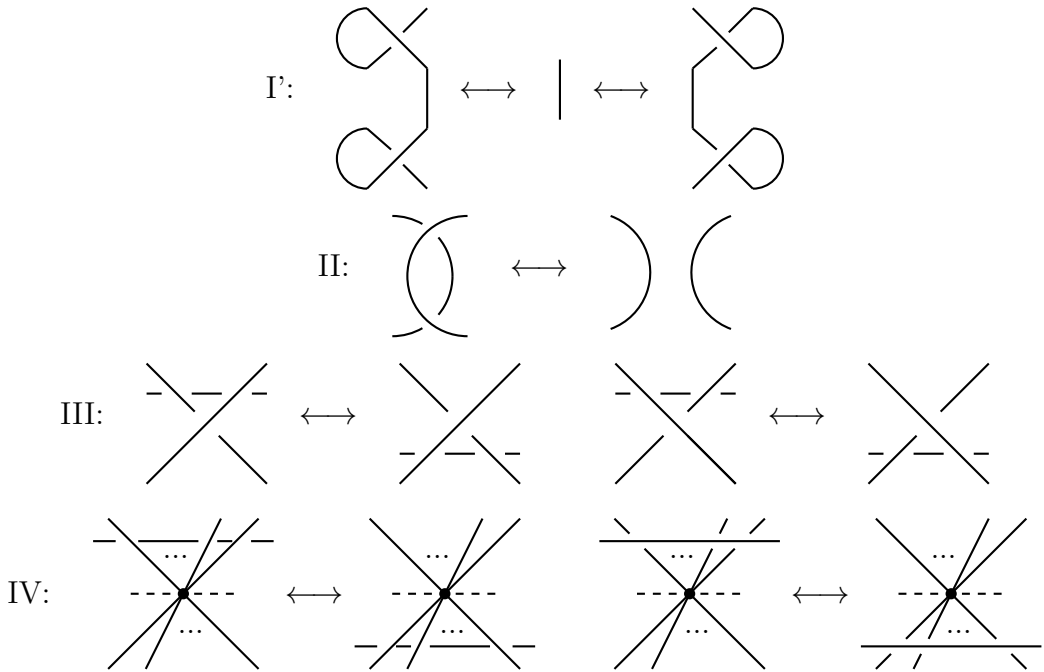

Our normalization idea is simple and natural: we construct a factor that cancels the change in the state sum arising from Reidemeister move (I') in the framed setting (which corresponds to move (I) in the unframed setting), while preserving invariance under the other moves. 
Henceforth, we denote by blackboard-bold letters $\mathbb{G}$ and $\mathbb{D}$ a framed transverse graph and its diagram, respectively, and by $G$ and $D$ the corresponding graph and diagram which disregard the framing.

\begin{defn}\label{def:normalizedAlexanderpolynomial}
\rm
For a framed transverse graph diagram $\mathbb{D}$, define the {\it normalized Alexander polynomial} by
\begin{equation}
\Delta_{\mathbb{D}}:=-\mathrm{Rot}(D)^{1/2}\cdot \langle D\rangle.
\end{equation}
Algebraically, this invariant takes values in the group ring $\mathbb{Z}[H_1(S^3\setminus G; \mathbb{Z})]$ (extended by formal square roots of the generators).
If we choose a set of generators $t_1, t_2, \ldots, t_k$ of $H_1(S^3\setminus G; \mathbb{Z})$ corresponding to the oriented meridians of some edges of $G$,  
then $\Delta_{\mathbb{D}}$ becomes a Laurent polynomial in $t_1^{1/2}, t_2^{1/2}, \ldots, t_k^{1/2}$, denoted $\Delta_{\mathbb{D}}(t_1, t_2, \ldots, t_k)$.
\end{defn}

\begin{rem}
The reason for inserting a minus sign in the definition of $\Delta_{\mathbb{D}}$ is to ensure that the trivial knot, with oriented meridian $t$, satisfies $\Delta_{\mathbb{D}}=\frac{1}{t^{1/2}-t^{-1/2}}$. 
\end{rem}

The topological invariance of $\Delta_{\mathbb{D}}$ will be established in Section 4. Once proven, we denote this invariant by $\Delta_{\mathbb{G}}$ to indicate its independence from the choice of diagram representing the framed transverse graph $\mathbb{G}$.
\medskip

\subsection{Specialization to MOY graphs}

For an MOY graph $(G, c)$, there exists a homomorphism 
$
\phi_c: H_1(S^3\setminus G; \mathbb{Z}) \rightarrow \mathbb{Z}
$
that sends the oriented meridian of each edge $e$ to $c(e)$.  
The state sum $\langle D, c \rangle$ defined in \cite[Definition~2.11]{MR4090586} is a one-variable Laurent polynomial in $t^{1/2}$, obtained from $\langle D \rangle$ by applying $\phi_c$ to every element of $H_1(S^3\backslash G; \mathbb{Z})$.  

More precisely, we extend $\phi_c$ to a ring homomorphism $$\Phi_c: \mathbb{Z}H_1(S^3\setminus G; \mathbb{Z}) \to \mathbb{Z}[t^{\pm 1}]$$ as follows:
For any element $s \in H_1(S^3\setminus G; \mathbb{Z})$ with $\phi_c(s) = i \in \mathbb{Z}$, we define $\Phi_c(s) = t^i$. This definition extends linearly to the entire group ring. Under this extension, each occurrence of $s^{\pm 1/2}$ in $\langle D \rangle$ is mapped to $t^{\pm i/2}$ in $\langle D, c \rangle$.
We thus have:
\begin{equation}
\label{sumrelation}
\langle D, c \rangle = \Phi_c(\langle D \rangle).
\end{equation}
In particular, if $\langle D, c \rangle \neq 0$, then  $\langle D \rangle \neq 0$.

For a trivalent MOY graph diagram $(\mathbb{D}, c)$ equipped with a blackboard framing, we defined in \cite[Definition 3.7]{MR4090586} the normalized Alexander polynomial by
\[
\Delta_{(\mathbb{G}, c)}(t) := \frac{\langle D, c \rangle}{(t^{-1/2} - t^{1/2})^{|V|-1}} \cdot \mathcal{C}(D, c),
\]
where $|V|$ denotes the number of vertices of $\mathbb{D}$.  Note that in \cite[Definition 3.4]{MR4090586} we introduced a factor $\mathcal{F}(\mathbb{D}, c)$ to account for twists in the diagram. However, since we only consider blackboard framings here, this factor can be omitted. 

\begin{prop}
\label{relation1}
For a framed trivalent MOY graph $(\mathbb{G}, c)$, let $\Phi_c$ be the ring homomorphism defined above. Then 
\begin{equation}
\Phi_c(\Delta_{\mathbb{G}})=\Delta_{(\mathbb{G}, c)}(t)\cdot (t^{1/2}-t^{-1/2})^{|V|-1}.
\end{equation}
\end{prop}
\begin{proof}
Note that for a trivalent graph, $|V|$ is always an even number, so $(t^{1/2}-t^{-1/2})^{|V|-1}=-(t^{-1/2}-t^{1/2})^{|V|-1}$. The result then follows from Equation~(\ref{sumrelation}) and Proposition~\ref{rotrelation}.
\end{proof}

\section{Topological invariance}

We now establish the topological invariance of our normalized Alexander polynomial. We begin by characterizing plane graphs with non-zero state sum.
Recall that in graph theory, an oriented graph $G$ is called {\it strongly connected} if for every pair of vertices there exists a directed path connecting them. Several equivalent characterizations exist:

\begin{lemma}\label{stronglyconnected}
Let $G$ be an oriented connected graph. The following conditions are equivalent:
\begin{enumerate}
\item $G$ is strongly connected: there exists a directed path between each pair of vertices.
\item Every edge of $G$ is contained in a directed cycle.
\item There exists a positive integer coloring $c: E \to \mathbb{N}$ such that for each vertex $v$,
$$\sum_{\text{$e$: pointing into $v$}} c(e)=\sum_{\text{$e$: pointing out of $v$}} c(e).$$
    \end{enumerate}
   
\end{lemma}

\begin{proof}
We prove the equivalences:

$(i) \Rightarrow (ii)$: Suppose $G$ is strongly connected and let $e$ be an edge from vertex $u$ to $v$. Since $G$ is strongly connected, there exists a directed path from $v$ back to $u$. Concatenating $e$ with this path yields a directed cycle containing $e$.

$(ii) \Rightarrow (iii)$: We begin with an integer coloring that satisfies the balanced conditions at each vertex (such as assigning $c(e)=0$ for all edges $e\in E$). Then, for any edge $e$ with non-positive coloring, we select a directed cycle containing $e$ (which exists by assumption) and add a sufficiently large positive integer to the color of every edge in that cycle. This operation preserves the balance conditions at all vertices while increasing the colors along the cycle. By repeating this process for each edge with non-positive color, we eventually obtain a desired positive integer coloring.

$(iii) \Rightarrow (i)$: Suppose there exists a positive integer coloring $c$ satisfying the vertex balance conditions. We will show that $G$ must be strongly connected. For contradiction, assume there exist vertices $u$ and $v$ with no directed path from $u$ to $v$. Let $S$ be the set of vertices reachable from $u$ by directed paths, and $T$ be its complement (containing $v$). Consider the cut between $S$ and $T$. By construction, all edges crossing this cut must go from $T$ to $S$. The positive flow across this cut from $T$ to $S$ with no return flow violates the balance condition for the vertices in $S$ (and $T$) collectively, contradicting the existence of our coloring. Therefore, $G$ must be strongly connected.
\end{proof}

\begin{lemma}
\label{positive}
Let $D$ be a plane transverse graph diagram. If $\langle D\rangle_\delta \neq 0$ for all choices of base point $\delta$, then the underlying directed graph of $D$ is strongly connected.
\end{lemma}

\begin{proof}
We first note that for a plane graph diagram, namely a diagram without crossings of type \diaCrossP or \diaCrossN, the connectivity of the diagram $D$ as a subset of $\mathbb{R}^2$ coincides with the connectivity of the underlying graph. If $D$ were disconnected as a diagram, then $\langle D\rangle_\delta = 0$ for all $\delta$, contradicting our assumption. Thus, the graph must be connected.

Now, to prove strong connectivity, we show that every edge is contained in a directed cycle. Let $e$ be any edge of $D$ with head vertex $v$. Place the base point $\delta$ on $e$. By \cite[Theorem~3.4]{MR4300448}, the set of Kauffman states $S(D, \delta)$ corresponds bijectively to the set of oriented spanning trees of $D$ rooted at $v$, where edges are oriented away from $v$. Since $\langle D\rangle_\delta \neq 0$, there exists at least one such oriented spanning tree $T$.

Let $u$ be the tail vertex of $e$. Since $T$ is a spanning tree rooted at $v$, there exists a unique directed path $P$ in $T$ from $v$ to $u$. The concatenation of $P$ with $e$ forms a directed cycle containing $e$. As $e$ was arbitrary, every edge lies in a directed cycle, so $D$ is strongly connected by Lemma~\ref{stronglyconnected}.

\end{proof}

We now extend our analysis from plane graphs to general spatial graphs. The following proposition establishes base point independence for arbitrary transverse graphs.

\begin{prop}
\label{initial}
The state sum $\langle D\rangle_\delta$ does not depend on the choice of the base point $\delta$.
\end{prop}
\begin{proof}
The proof follows the approach of \cite{MR4090586}. If $D$ is disconnected (as a diagram), then $\langle D\rangle_\delta=0$ by definition, so we assume $D$ is connected.

Define an Alexander matrix $A(D)$ similar to \cite[Definition 2.14]{MR4090586}, using the local contribution $A_{C_p}^{\triangle}$ defined in Fig.~\ref{fig:e1}.  Let $A(D)\backslash(u, v)$ denote the square matrix obtained from $A(D)$ by removing the columns corresponding to the marked regions $R_u$ and $R_v$ adjacent to $\delta$. Then by \cite[Proposition 2.16]{MR4090586}, we have
$$\langle D\rangle_\delta= \pm \vert \delta \vert^{-1} \cdot \det A(D)\backslash(u, v).$$
Now let $\delta'$ be another choice of base point.
Using the same argument as in \cite[Proposition 2.18]{MR4090586}, we can show that $$\langle D\rangle_{\delta}=\pm \langle D \rangle_{\delta'}.$$ 

To resolve the sign ambiguity, we first consider plane graph diagrams.  
Since $\langle D \rangle_{\delta} = 0$ if and only if $\langle D \rangle_{\delta'} = 0$, we may assume both are nonzero.  
By Lemma~\ref{positive}, $D$ is strongly connected, so there exists a positive integer coloring $c$ of $D$. 
Then for the MOY graph $(D, c)$, by \cite[Proposition~2.18]{MR4090586}, 
$$
\langle D, c \rangle_{\delta} = \langle D, c \rangle_{\delta'}.
$$
Furthermore, since $(D, c)$ is a plane graph with a positive integer coloring,  
\cite[Section~5.4]{MR4090586} implies that $\langle D, c \rangle \neq 0$.  
Recalling that $\Phi_c(\langle D \rangle) = \langle D, c \rangle$, we conclude that
$$
\langle D \rangle_{\delta} = \langle D \rangle_{\delta'}.
$$

For general graph diagrams containing crossings of type \diaCrossP or \diaCrossN, we extend this result using skein-type relations in Lemma \ref{skein}.
Note that these relations hold for an arbitrary choice of $\delta$ as long as it is consistently placed in the same position on both sides. The proposition, which holds for plane graphs, thus extends to general graph diagrams.
\end{proof}

\begin{theo}\label{thmDeltaG}
The diagram-defined polynomial $\Delta_{\mathbb{D}}$ is a topological invariant of the framed transverse graph $\mathbb{G}$, which we denote by $\Delta_{\mathbb{G}}$.
\end{theo}
\begin{proof}
Since both $\langle D\rangle$ and $\mathrm{Rot}(D)$ are invariant under Reidemeister moves $(II), (III)$ and $(IV)$, it remains to check move $(I')$. Suppose the graphs on the two sides of $(I')$ are $\mathbb{D}_1$ and $\mathbb{D}_2$, respectively. By Proposition~\ref{invtheo}, we have $\langle D_1\rangle = t^{-1}\langle D_2\rangle$, where $t$ is the oriented meridian of the edge. On the other hand, Proposition~\ref{rotmove1} gives $\mathrm{Rot}(D_1) = t^2 \mathrm{Rot}(D_2)$. Hence, under move $(I')$ we obtain
\[
\Delta_{\mathbb{D}_1} = \mathrm{Rot}(D_1)^{1/2} \cdot \langle D_1\rangle = (t^2\, \mathrm{Rot}(D_2))^{1/2} \cdot (t^{-1}\langle D_2\rangle) = \mathrm{Rot}(D_2)^{1/2} \cdot \langle D_2\rangle = \Delta_{\mathbb{D}_2}.
\]
Combined with base point independence from Proposition~\ref{initial}, this shows $\Delta_{\mathbb{D}}$ depends only on the framed transverse graph $\mathbb{G}$, justifying the notation $\Delta_{\mathbb{G}}$.

\end{proof}

\begin{rem}
While $\Delta_{\mathbb{G}}$ is a topological invariant of the framed transverse graph $\mathbb{G}$ as an element of the group ring $\mathbb{Z}[H_1(S^3\setminus \mathbb{G}; \mathbb{Z})]$, 
its expression as a multi-variable Laurent polynomial depends on the choice of generators for the presentation of $H_1(S^3\setminus \mathbb{G}; \mathbb{Z})$. 
Different choices of generators will yield different but equivalent polynomial representations related by the appropriate change of variables.
\end{rem}

\section{Relation with Viro's $U_q(\mathfrak{gl}(1\vert 1))$-Alexander polynomial}\label{section:viro}

\subsection{Viro's $U_q(\mathfrak{gl}(1\vert 1))$-Alexander polynomial}

Viro \cite{MR2255851} constructed a multi-variable Alexander polynomial $\underline{\Delta}^1(\mathbb{G})$ for framed trivalent graphs using representation theory of finite-dimensional modules over $U^1$, a subalgebra of the $q$-deformed universal enveloping superalgebra $U_{q}(\mathfrak{gl}(1 \vert 1))$. We outline the key aspects relevant to our work.

For an oriented framed trivalent graph $\mathbb{G}$ without sinks or sources, let  
\[
M = H^1(\mathbb{G}; \mathbb{Z}) \cong H_1(S^3 \setminus \mathbb{G}; \mathbb{Z}),
\]
and let $B$ be the quotient ring of $\mathbb{Z}[M]$.  
Each edge of $\mathbb{G}$ is assigned with an element $t \in M$, called \textit{multiplicity}, which corresponds to the homology class of the oriented meridian of that edge under the above identification.  

\begin{rem}
In Viro's original construction, each edge carries both a multiplicity $t \in M$ and a weight $N \in \mathbb{Z}$, corresponding to modules over $U^1 \otimes_{\mathbb{Z}} B$. Throughout this paper, we set $N = 1$ unless otherwise specified.
\end{rem}

In the above setup, we now describe a procedure for computing $\underline{\Delta}^1(\mathbb{G})$. 
Choose a diagram $\mathbb{D}$ of $\mathbb{G}$ in $\mathbb{R}^2$. 
The diagram divides $\mathbb{R}^2$ into several regions, one of which is unbounded. 
Select an edge of $\mathbb{D}$ on the boundary of the unbounded region, and cut it at a generic point $\delta$. 
Suppose the oriented meridian of this edge is $t$. 

Next, deform the diagram by isotopies of $\mathbb{R}^2$ into Morse position. 
Arrange it so that the two endpoints obtained from cutting $\delta$ have heights $0$ and $1$, and all critical points, crossings, half-twist symbols, and vertices of the diagram occur at distinct heights between $0$ and $1$. 

In other words, after this deformation, the diagram can be divided by horizontal lines into several layers, each consisting of a disjoint union of trivial vertical segments together with one of the eight \textit{elementary pieces} shown in Fig.~\ref{fig2}. 
For diagrams with blackboard framing, the half-twist symbols do not appear.

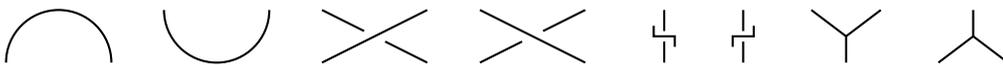
\begin{figure}[!h]
\begin{tikzpicture}[baseline=-0.65ex, thick, scale=0.7]
\draw (-1, 0) arc (0:180:1);
\draw (0, 1) arc (180:360:1);
\draw (3, 0) -- (5, 1);
\draw (5, 0) -- (4.2, 0.4);
\draw (3, 1) -- (3.8, 0.6);
\draw (8, 0) -- (6, 1);
\draw (8, 1) -- (7.2, 0.6);
\draw (6, 0) -- (6.8, 0.4);
\draw (9.5, 0) -- (9.5, 0.4);
\draw (9.5, 1) -- (9.5, 0.6);
\draw (9.3, 0.7) -- (9.3, 0.5) -- (9.7, 0.5) --(9.7, 0.3);
\draw (11, 0) -- (11, 0.4);
\draw (11, 1) -- (11, 0.6);
\draw (10.8, 0.3) -- (10.8, 0.5) -- (11.2, 0.5) --(11.2, 0.7);
\end{tikzpicture}\hspace{6mm}
\begin{tikzpicture}[baseline=-0.65ex, thick, scale=0.7]
\draw  (0, 0)  to (0,0.5);
\draw   (0,0.5)   to  (0.66,1);
\draw   (0,0.5)  to  (-0.66,1);
\end{tikzpicture}\hspace{6mm}
\begin{tikzpicture}[baseline=-0.65ex, thick, scale=0.7]
\draw  (0, 0.5)  to (0,1);
\draw   (0,0.5)   to  (0.66,0);
\draw   (0,0.5)   to  (-0.66,0);
\end{tikzpicture}
\caption{Critical points, crossings, half-twist symbols and vertices.}
\label{fig2}
\end{figure}

Each elementary piece in the Morse-position diagram connects sequences of endpoints from bottom to top. Viro's functor assigns to each such piece a morphism between tensor products of 2-dimensional vector spaces, with the morphisms determined by Boltzmann weights from Tables 3 and 4 of \cite{MR2255851}.


\begin{prop}[Vertex state sum representation in \cite{MR2255851}]
\label{sum} 
Let $e_0$ (boson) and $e_1$ (fermion) denote the two basis elements of a 2-dimensional vector space.
The invariant $\underline{\Delta}^1(\mathbb{G})$ can be calculated as follows: 

\begin{enumerate}
\item A state is a map 
$$s: E\to \{e_0, e_1\}$$
such that the edge with the point $\delta$ is mapped to $e_1$.
Graphically, edges assigned $e_0$ are drawn as dotted lines, and those assigned $e_1$ as solid lines.

\item For each state, compute the product of Boltzmann weights (from Tables 3--4 of \cite{MR2255851}) at all special points: critical points, the crossings, the half-twist symbols, and the vertices. This gives a value $m(s) \in B$ for each $s$.

\item Sum $m(s)$ over all states and multiply by $\displaystyle \frac{1}{t^2-t^{-2}}$ to obtain $\underline{\Delta}^1(\mathbb{G})$.
\end{enumerate}
\end{prop}

The Boltzmann weights enforce specific constraints at vertices. In particular, \cite[Table 4]{MR2255851} shows that for vertices which are neither sinks nor sources, only the following edge assignments yield nonzero contributions:

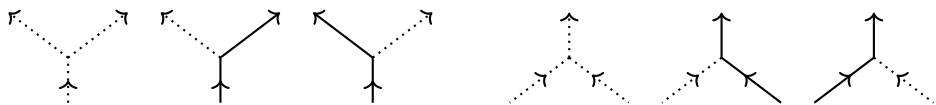
\begin{align*}
\begin{tikzpicture}[baseline=-0.65ex, thick, scale=1.2]
\draw [dotted] (0, 0) [->-] to (0,0.5);
\draw  [dotted] (0,0.5) [->]  to  (0.66,1);
\draw  [dotted] (0,0.5) [->]  to  (-0.66,1);
\end{tikzpicture} \quad
\begin{tikzpicture}[baseline=-0.65ex, thick, scale=1.2]
\draw  (0, 0) [->-] to (0,0.5);
\draw   (0,0.5) [->]  to  (0.66,1);
\draw  [dotted] (0,0.5) [->]  to  (-0.66,1);
\end{tikzpicture} \quad
\begin{tikzpicture}[baseline=-0.65ex, thick, scale=1.2]
\draw  (0, 0) [->-] to (0,0.5);
\draw  [dotted] (0,0.5) [->]  to  (0.66,1);
\draw   (0,0.5) [->]  to  (-0.66,1);
\end{tikzpicture} \hspace{1cm}
\begin{tikzpicture}[baseline=-0.65ex, thick, scale=1.2]
\draw [dotted] (0, 0.5) [->] to (0,1);
\draw  [dotted] (0,0.5) [-<-]  to  (0.66,0);
\draw [dotted]  (0,0.5) [-<-]  to  (-0.66,0);
\end{tikzpicture} \quad
\begin{tikzpicture}[baseline=-0.65ex, thick, scale=1.2]
\draw  (0, 0.5) [->] to (0,1);
\draw   (0,0.5) [-<-]  to  (0.66,0);
\draw  [dotted] (0,0.5) [-<-]  to  (-0.66,0);
\end{tikzpicture} \quad
\begin{tikzpicture}[baseline=-0.65ex, thick, scale=1.2]
\draw  (0, 0.5) [->] to (0,1);
\draw  [dotted] (0,0.5) [-<-]  to  (0.66,0);
\draw   (0,0.5) [-<-]  to  (-0.66,0);
\end{tikzpicture}
\end{align*}

Therefore, for a plane graph $\mathbb{G}$ without sinks or sources, solid edges in each state form disjoint simple closed curves, one containing $\delta$.  The following example illustrates this situation.

\begin{example}
Consider the plane graph $\Gamma$ in Fig. \ref{Fig:gamma} with $\delta$ paced on the edge $e$. The states that can potentially have a nonzero $m(s)$ are listed in Fig.~\ref{Fig:gamma2}.
\begin{figure}[h!]
\begin{tikzpicture}[baseline=-0.65ex, thick, scale=1]
\draw (0, 2) node[circle,fill,inner sep=1.5pt]{};
\draw (0, 1) node[circle,fill,inner sep=1.5pt]{};
\draw (1, 0) node[circle,fill,inner sep=1.5pt]{};
\draw (-1, 0) node[circle,fill,inner sep=1.5pt]{};
\draw (0, 1) [->-] to (0, 2);
\draw (1, 0) [->-] to (0, 1);
\draw (-1, 0) [->-] to (0, 1);
\draw (-1, 0) [->-] to (1, 0);
\draw (0,2)  [-<-] to [out=30,in=45] (1,0);
\draw (0,2)  [->-] to [out=150,in=135] (-1,0);
\draw (0, 2.2) node {$v_1$};
\draw (-1.2, -0.2) node {$v_2$};
\draw (1.3, -0.2) node {$v_3$};
\draw (0.3, 1) node {$v_4$};
\draw (-1.3, 1.2) node {$e$};
\end{tikzpicture}
\caption{Graph $\Gamma$.}
\label{Fig:gamma}
\end{figure}
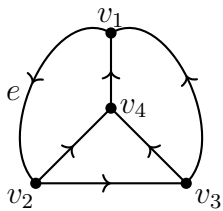

\begin{figure}[h!]
\begin{tikzpicture}[baseline=-0.65ex, thick, scale=1]
\draw (0, 2) node[circle,fill,inner sep=1.5pt]{};
\draw (0, 1) node[circle,fill,inner sep=1.5pt]{};
\draw (1, 0) node[circle,fill,inner sep=1.5pt]{};
\draw (-1, 0) node[circle,fill,inner sep=1.5pt]{};
\draw (0, 1) [->-] to (0, 2);
\draw (1, 0) [->-] to (0, 1);
\draw [dotted](-1, 0) [->-] to (0, 1);
\draw (-1, 0) [->-] to (1, 0);
\draw [dotted](0,2)  [-<-] to [out=30,in=45] (1,0);
\draw (0,2)  [->-] to [out=150,in=135] (-1,0);
\draw (0, 2.2) node {$v_1$};
\draw (-1.2, -0.2) node {$v_2$};
\draw (1.3, -0.2) node {$v_3$};
\draw (0.3, 1) node {$v_4$};
\draw (-1.3, 1.2) node {$e$};
\end{tikzpicture}\quad\quad
\begin{tikzpicture}[baseline=-0.65ex, thick, scale=1]
\draw (0, 2) node[circle,fill,inner sep=1.5pt]{};
\draw (0, 1) node[circle,fill,inner sep=1.5pt]{};
\draw (1, 0) node[circle,fill,inner sep=1.5pt]{};
\draw (-1, 0) node[circle,fill,inner sep=1.5pt]{};
\draw (0, 1) [->-] to (0, 2);
\draw [dotted](1, 0) [->-] to (0, 1);
\draw (-1, 0) [->-] to (0, 1);
\draw [dotted](-1, 0) [->-] to (1, 0);
\draw [dotted](0,2)  [-<-] to [out=30,in=45] (1,0);
\draw (0,2)  [->-] to [out=150,in=135] (-1,0);
\draw (0, 2.2) node {$v_1$};
\draw (-1.2, -0.2) node {$v_2$};
\draw (1.3, -0.2) node {$v_3$};
\draw (0.3, 1) node {$v_4$};
\draw (-1.3, 1.2) node {$e$};
\end{tikzpicture}\quad\quad
\begin{tikzpicture}[baseline=-0.65ex, thick, scale=1]
\draw (0, 2) node[circle,fill,inner sep=1.5pt]{};
\draw (0, 1) node[circle,fill,inner sep=1.5pt]{};
\draw (1, 0) node[circle,fill,inner sep=1.5pt]{};
\draw (-1, 0) node[circle,fill,inner sep=1.5pt]{};
\draw [dotted](0, 1) [->-] to (0, 2);
\draw [dotted](1, 0) [->-] to (0, 1);
\draw [dotted](-1, 0) [->-] to (0, 1);
\draw (-1, 0) [->-] to (1, 0);
\draw (0,2)  [-<-] to [out=30,in=45] (1,0);
\draw (0,2)  [->-] to [out=150,in=135] (-1,0);
\draw (0, 2.2) node {$v_1$};
\draw (-1.2, -0.2) node {$v_2$};
\draw (1.3, -0.2) node {$v_3$};
\draw (0.3, 1) node {$v_4$};
\draw (-1.3, 1.2) node {$e$};
\end{tikzpicture}
\caption{Three states contributing to $\underline{\Delta}^1(\Gamma)$.}
\label{Fig:gamma2}
\end{figure}
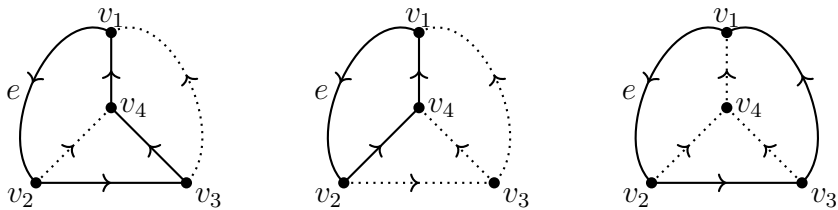
\end{example}

From this observation, we derive:
\begin{lemma}
\label{plane1}
For an oriented trivalent plane graph $\mathbb{G}$ without sinks or sources, if $\underline{\Delta}^1(\mathbb{G}) \neq 0$, then $\mathbb{G}$ is strongly connected.
\end{lemma}
\begin{proof}

Since $\mathbb{G}$ is a plane graph, non-connectivity would imply $\underline{\Delta}^1(\mathbb{G}) = 0$, contradicting our assumption. Thus, $\mathbb{G}$ is connected.

For strong connectivity, we show every edge lies in a directed cycle. Let $e$ be an arbitrary edge connecting vertices $v_1$ and $v_2$. Place the point $\delta$ on $e$. Since $\underline{\Delta}^1(\mathbb{G}) \neq 0$, there exists a state $s$ with $m(s) \neq 0$ as in Proposition~\ref{sum}. In this state, the edge $e$ must be assigned $e_1$ (solid). The solid edges in state $s$ form disjoint oriented cycles, one of which contains $e$. This cycle provides a directed path from $v_2$ back to $v_1$, demonstrating that $e$ lies in a directed cycle. As $e$ was arbitrary, $\mathbb{G}$ is strongly connected.

\end{proof}

For a framed MOY graph $(\mathbb{G},c)$ with positive integer coloring $c$, recall from Section 3.4 the ring homomorphism $\Phi_c: \mathbb{Z}H_1(S^3 \setminus \mathbb{G}; \mathbb{Z}) \to \mathbb{Z}[t^{\pm 1}]$ defined by sending the meridian generator $t_e$ to $t^{c(e)}$. Applying this homomorphism to Viro's polynomial yields the single-variable specialization
\begin{equation}
\label{relation2}
\underline{\Delta}^1(\mathbb{G}, c) = \Phi_c(\underline{\Delta}^1(\mathbb{G})).
\end{equation}

\subsection{Equivalence of invariants}

We now establish the equivalence between our normalized Alexander polynomial and Viro's invariant. The connection was first observed in the context of MOY graphs:

\begin{theo}[\cite{MR4001658}]
\label{oldtheo}
For a framed trivalent MOY graph $(\mathbb{G}, c)$, 
we have
\begin{equation}
\label{single}
\Delta_{(\mathbb{G}, c)}(t^{4}) =\displaystyle \frac{\prod_{\text{$v$: even type}} (t^{2c(v)}-t^{-2c(v)}) }{(t^{2}-t^{-2})^{\vert V \vert-1}}\underline{\Delta}^1(\mathbb{G}, c),
\end{equation}
where the product is over vertices of even type (two incoming, one outgoing edge), and $c(v)$ denotes the color of the outgoing edge.

\end{theo}

\begin{rem}
On the left hand side of \eqref{single}, the variable $q^{-1}$ in \cite[Theorem~3.2]{MR4001658} is replaced by $t$ due to a variable change between the preprint and published version of \cite{MR4090586}.
\end{rem}

Let $t_1, \ldots, t_k$ be generators of $H_1(S^3\setminus \mathbb{G}; \mathbb{Z})$.
Using the ring homomorphism $\Phi_c$ and Proposition \ref{relation1}, we can reformulate Theorem \ref{oldtheo} as an identity relating the single-variable specializations of the two multi-variable Alexander polynomials:

\begin{prop}
For any positive rational coloring $c$:
\begin{equation} \label{MultitoSingle}
\Phi_c\left(\Delta_{\mathbb{G}}(t_1^4, t_2^4, \ldots, t_k^4)\right)=\Phi_c \left( \prod_{\text{$v$: even type}} (t_v^{2}-t_v^{-2}) \cdot \underline{\Delta}^1(\mathbb{G}) \right),
\end{equation}
where $t_v$ is the oriented meridian of the outgoing edge around $v$. 
\end{prop}

\begin{proof}
While Theorem \ref{oldtheo} establishes this identity for positive integer colorings, we can extend it to all positive rational colorings through a scaling argument. Given a positive rational coloring $c$, there exists a positive integer $N$ such that $Nc$ is an integer coloring. Applying Theorem \ref{oldtheo} to $Nc$ and then making the variable substitution $t \rightarrow t^{1/N}$, we obtain the identity for the rational coloring $c$.
\end{proof}

This result suggests a deeper relationship between these polynomials. Indeed, we can strengthen it to show that the multi-variable Alexander polynomials themselves satisfy a corresponding identity.

\begin{theo}
\label{main}
Let $\mathbb{G}$ be a framed trivalent graph without sinks or sources, and let $t_1, \ldots, t_k$ be generators of $H_1(S^3\setminus \mathbb{G}; \mathbb{Z})$. Then the normalized Alexander polynomial $\Delta_{\mathbb{G}}$ is related to Viro's $U_q(\mathfrak{gl}(1\vert 1))$-Alexander polynomial $\underline{\Delta}^1(\mathbb{G})$ by:
\begin{equation}
\label{maineq}
\Delta_{\mathbb{G}}(t_1^4, t_2^4, \ldots, t_k^4)=\prod_{\text{$v$: even type}} (t_v^{2}-t_v^{-2}) \cdot \underline{\Delta}^1(\mathbb{G}).
\end{equation}
\end{theo}

The extension from the specialization identity to the multi-variable identity follows a fundamental principle in algebraic geometry: if two polynomials agree on a sufficiently dense set of points, they must be identical. In our case, the identity holds for all positive rational colorings, which form a dense subset in the space of all possible evaluations. For a rigorous proof, we can rely on the following classical result about polynomials:

\begin{lemma}[Corollary IV.1.6 in \cite{Lang}]
\label{Lang}
Let $\mathbb{F}$ be a field, and let $S_1, \cdots, S_n$ be infinite subsets of $\mathbb{F}$.  Let $f(x_1, \cdots, x_n)$ be a polynomial in $n$ variables over $\mathbb{F}$.  If $f(a_1, \cdots, a_n)=0$ for all choices $a_i\in S_i$ ($i=1, \cdots, n$), then $f=0$.
\end{lemma}

\begin{lemma} 
\label{plane}
The invariants on both sides of (\ref{maineq}) satisfy the following relations.
\begin{align}
&\left(\begin{tikzpicture}[baseline=-0.65ex, thick, scale=0.5]
\draw (1,-1)   -- (0.2,-0.2);
\draw (-1, -1) [->] -- (1, 1) node[above]{$t$};
\draw (-0.2,0.2) [->] to (-1,1)  node[above]{$s$};
\end{tikzpicture} \right)
  =\frac{-t^{-2}s^{-2}}{\{2\}_t\{2\}_s}\cdot
\left(
\begin{tikzpicture}[baseline=-0.65ex, thick, scale=0.5]
\draw (0,-2) node[below]{$t$} [->-] to (0, 0);
\draw (0,0) --  (0, 1);
\draw (0, 1) [->] to (0,2) node[above]{$s$};
\draw (0,1) node[circle,fill,inner sep=1pt]{};
\draw (2,0) node[circle,fill,inner sep=1pt]{};
\draw (2,-2) [->-] node[below]{$s$} to (2,0);
\draw (2,0) [->] to (2,2) node[above]{$t$};
\draw (2,-0) [->-] to (0, 1);
\draw (1, -0.3) node {$st^{-1}$};
\end{tikzpicture}\right)
+\, \frac{s^{-2}}{\{2\}_t\{2\}_{ts}}\cdot
\left (\begin{tikzpicture}[baseline=-0.65ex, thick]
\draw (0,-1) [->-] to  (0.5,-0.33);
\draw (0.5, -0.33) [->-] to  (0.5,0.33);
\draw (0.5, 0.33) [->] to  (0,1);
\draw (1,-1) [->-] to (0.5,-0.33);
\draw (1,1) [<-] to  (0.5,0.33);
\draw (0, -1.25) node {$t$};
\draw (0.9,-1.25) node {$s$};
\draw (0,1.25) node {$s$};
\draw (1,1.25) node {$t$};
\draw (1,0) node {$ts$};
\draw (0.5,0.33) node[circle,fill,inner sep=1pt]{};
\draw (0.5,-0.33) node[circle,fill,inner sep=1pt]{};
\end{tikzpicture}\right),\\
&\left(
\begin{tikzpicture}[baseline=-0.65ex, thick, scale=0.5]
\draw (1,-1) [->]  -- (-1,1)  node[above]{$s$};
\draw (-1, -1) to (-0.2, -0.2);
\draw (0.2, 0.2) [->] -- (1, 1) node[above]{$t$};
\end{tikzpicture}\right)
  =\frac{-t^{2}s^{2}}{\{2\}_t\{2\}_s}\cdot
\left(
\begin{tikzpicture}[baseline=-0.65ex, thick, scale=0.5]
\draw (0,-2) node[below]{$t$} [->-] to (0, 0);
\draw (0,0) --  (0, 1);
\draw (0, 1) [->] to (0,2) node[above]{$s$};
\draw (0,1) node[circle,fill,inner sep=1pt]{};
\draw (2,0) node[circle,fill,inner sep=1pt]{};
\draw (2,-2) [->-] node[below]{$s$} to (2,0);
\draw (2,0) [->] to (2,2) node[above]{$t$};
\draw (2,-0) [->-] to (0, 1);
\draw (1, -0.3) node {$st^{-1}$};
\end{tikzpicture}\right)
+ \, \frac{s^{2}}{\{2\}_t\{2\}_{ts}}\cdot
\left (\begin{tikzpicture}[baseline=-0.65ex, thick]
\draw (0,-1) [->-] to  (0.5,-0.33);
\draw (0.5, -0.33) [->-] to  (0.5,0.33);
\draw (0.5, 0.33) [->] to  (0,1);
\draw (1,-1) [->-] to (0.5,-0.33);
\draw (1,1) [<-] to  (0.5,0.33);
\draw (0, -1.25) node {$t$};
\draw (0.9,-1.25) node {$s$};
\draw (0,1.25) node {$s$};
\draw (1,1.25) node {$t$};
\draw (1,0) node {$ts$};
\draw (0.5,0.33) node[circle,fill,inner sep=1pt]{};
\draw (0.5,-0.33) node[circle,fill,inner sep=1pt]{};
\end{tikzpicture}\right),
\end{align}
where $\{2\}_u:=u^2-u^{-2}$ for $u\in H_1(S^3\setminus \mathbb{G}; \mathbb{Z})$.
\end{lemma}

\begin{proof}

The verification of the skein relations for our multi-variable Alexander polynomial $\Delta_\mathbb{G}(t_1^4, t_2^4, \ldots, t_k^4)$ follows the same pattern as the proof of \cite[Theorem 4.1(iv)]{MR4090586}, comparing the Kauffman states and their state sums on both sides of each relation. 
For Viro's invariant, this can be established by a straightforward application of the Boltzmann weights listed in Tables 3–4 of \cite{MR2255851}. The argument follows the same idea as the proof of \cite[Theorem 2.4 (viii)]{MR4001658}, where a similar relation was established for the single-variable specialization $\underline{\Delta}^1(\mathbb{G}, c)$.

\end{proof}

\begin{proof}[Proof of Theorem \ref{main}]
By Lemma \ref{plane}, it suffices to prove the equation for plane graphs $\mathbb{G}$. If $\mathbb{G}$ is not strongly-connected, Lemmas \ref{positive} and \ref{plane1} show both $\Delta_{\mathbb{G}}$ and $\underline{\Delta}^1(\mathbb{G})$ are zero, hence equal. Now we consider the case that $\mathbb{G}$ is strongly-connected. 

Let $k$ be the rank of $H_1(S^3 \setminus \mathbb{G}; \mathbb{Z})$, and assume $t_{1}, \cdots, t_{k}$ form a generating set, where $t_{i}$ is the oriented meridian of an edge $e_i$ for $1\leq i \leq k$. Each meridian $t_e$ of an edge $e$ can be expressed as a monomial in $t_{1}, \cdots, t_{k}$, so its color $c(e)$ is a linear combination of $c(e_1), \cdots, c(e_k)$ under a coloring $c$. 

Since $\mathbb{G}$ is strongly-connected, by Lemma \ref{initial}, there exists a positive integer coloring $c_0$. Since positivity is determined by finitely many strict linear inequalities $c(e) > 0$, there is an open neighborhood $U_i$ of $c_0(e_i)$ for $1\leq i \leq k$ 
such that every rational coloring $c$ with $(c(e_1), \cdots, c(e_k)) \in U_1\times \cdots U_k$ has positive rational colors on all edges of $\mathbb{G}$.

Define the polynomial $$f(t_{1}, \cdots, t_{k}):=\Delta_{\mathbb{G}}(t_1^4, t_2^4, \ldots, t_k^4)-\prod_{\text{$v$: even type}} (t_v^{2}-t_v^{-2}) \cdot \underline{\Delta}^1(\mathbb{G})$$ 
where each $t_e$ is expressed in terms of the basis elements $t_{1}, \ldots, t_{k}$. 

Let $S_i:=\{2^{a_i} \,|\, a_i\in U_i\cap \mathbb{Q}\}$ for $1\leq i \leq k$.  For any coloring $c$ with $c(e_i) = a_i$, equation (\ref{MultitoSingle}) gives
$$f(2^{a_1}, \cdots, 2^{a_k})=[\Phi_c(f(t_{1}, \cdots, t_{k}) ]_{t=2}=0.$$
Since each $S_i$ is an infinite subset of $\mathbb{R}$, applying Lemma \ref{Lang} with $\mathbb{F} = \mathbb{R}$ shows that $f = 0$, proving the theorem.

\end{proof}

\bibliographystyle{siam}
\bibliography{MultAlex}

\end{document}